\documentclass{amsart}
\usepackage{amsfonts, amsbsy, amsmath, amssymb}

\newtheorem{thm}{Theorem}[section]
\newtheorem{lem}[thm]{Lemma}

\numberwithin{equation}{section}

\begin{document}

\title[a Type of Permutation Trinomials over Finite Fields] {Determination of a Type of Permutation Trinomials over Finite Fields}

\author[Xiang-dong Hou]{Xiang-dong Hou*}
\address{Department of Mathematics and Statistics,
University of South Florida, Tampa, FL 33620}
\email{xhou@usf.edu}
\thanks{* Research partially supported by NSA Grant H98230-12-1-0245.}



\begin{abstract}
Let $f=a{\tt x} +b{\tt x}^q+{\tt x}^{2q-1}\in\Bbb F_q[{\tt x}]$. We find explicit conditions on $a$ and $b$ that are necessary and sufficient for $f$ to be a permutation polynomial of $\Bbb F_{q^2}$. This result allows us to solve a related problem. Let $g_{n,q}\in\Bbb F_p[{\tt x}]$ ($n\ge 0$, $p=\text{char}\,\Bbb F_q$) be the polynomial defined by the functional equation $\sum_{c\in\Bbb F_q}({\tt x}+c)^n=g_{n,q}({\tt x}^q-{\tt x})$. We determine all $n$ of the form $n=q^\alpha-q^\beta-1$, $\alpha>\beta\ge 0$, for which $g_{n,q}$ is a permutation polynomial of $\Bbb F_{q^2}$. 
\end{abstract}

\maketitle

\tableofcontents
%%%%%%%%%%%%%%%%%%%%%%%%%%%%%%%%%%%%%%%%%%%%%%%%%%%%%%%%%%%%%%%%%%%%%%%%%
%             section 1 
%%%%%%%%%%%%%%%%%%%%%%%%%%%%%%%%%%%%%%%%%%%%%%%%%%%%%%%%%%%%%%%%%%%%%%%%%
\section{Introduction}

Let $\Bbb F_q$ denote the finite field with $q$ elements. A polynomial $f\in\Bbb F_q[{\tt x}]$ is called a permutation polynomial (PP) of $\Bbb F_q$ if the mapping $x\mapsto f(x)$   
is a permutation of $\Bbb F_q$. Permutation polynomials over finite fields are studied for both theoretic \cite{Eva,Hir71,Hir75,LN,Nie-Rob} 
and practical \cite{Gol-Mor,Lai07,Lev-Bra,Lev-Cha} reasons. PPs with few terms (excluding monomials) are particularly sought after \cite{Akb-Wan06, Lee-Par97, Mas-Pan-Wan06, Mas-Zie09, Tur88, Wan87, Wan94, Zie09}.

In the present paper we consider trinomials of the form $f=a{\tt x}+b{\tt x}^q+{\tt x}^{2q-1}\in\Bbb F_q[{\tt x}]$. Since $f\equiv(a+b+1){\tt x}\pmod{{\tt x}^q-{\tt x}}$, $f$ is a PP of $\Bbb F_q$ if and only if $a+b+1\ne 0$. The question that we are interested in is when $f$ is a PP of $\Bbb F_{q^2}$. This question will be completely answered in Theorem A (for odd $q$) and Theorem B (for even $q$). Partial solutions to the question appeared in two recent papers: PPs of $\Bbb F_{q^2}$ the form $t{\tt x}+{\tt x}^{2q-1}$ ($t\in\Bbb F_{q^*}$) and of the form $-{\tt x}+t{\tt x}^q+{\tt x}^{2q-1}$ ($t\in\Bbb F_{q^*}$) were determined in \cite{Hou1} and \cite{Hou2}, respectively.  For the proofs of the Theorems A and B, we draw on the methods of \cite{Hou1} and \cite{Hou2}, especially, the approach of \cite{Hou2}. However, the proofs in the present paper are much more than a routine adaptation of the ones in \cite{Hou1,Hou2}. We find a new method for proving the uniqueness of a solution $x\in\Bbb F_{q^2}$   
of the equation $ax+bx^q+x^{2q-1}=y$, where $y\in\Bbb F_{q^2}$. A common theme throughout the proofs of Theorems A and B is that complicated computations that appear to be heading nowhere can produce surprisingly nice results. For example, a seemingly out-of-control polynomial of degree 4 not only factors but factors exactly the way we desire; see \eqref{3.12}. 

Theorem A  provides a solution to a related problem. For each integer $n\ge 0$, let $g_{n,q}\in\Bbb F_q[{\tt x}]$ ($p=\text{char}\,\Bbb F_q$) be the polynomial defined by the functional equation
\[
\sum_{c\in\Bbb F_q}({\tt x}+c)^n=g_{n,q}({\tt x}^q-{\tt x}).
\]
The permutation property of the polynomial $g_{n,q}$ was the focus of several recent papers \cite{FHL,Hou11, Hou12}. These studies have led to the discovery of many new interesting PPs including the ones in \cite{Hou1,Hou2} and in the present paper. The ultimate goal concerning $g_{n,q}$ is to determine all triples of integers $(n,e;q)$ for which $g_{n,q}$ is a PP of $\Bbb F_{q^e}$; we call such triples {\em desirable}. While this goal may be out of reach for the time being, significant progress has been made. It was observed through computer search that many desirable triples appear in the form $(q^\alpha-q^\beta-1,2;q)$, where $\alpha>\beta\ge 0$. However, the chaotic values of those $\alpha$ and $\beta$ were quite bewildering; see \cite[Section 5 and Table 1]{FHL}. In the present paper, we are able to determine all desirable triples of this form; the results are stated in Theorems C (for even $q$) and D (for odd $q$). 
Theorem~C is an immediate consequence of some existing results. For Theorem~D, we note that
when $n=q^\alpha-q^\beta-1$, the polynomial $g_{n,q}$, modulo ${\tt x}^{q^2}-{\tt x}$, can be transformed through an invertible change of variable into the form $A{\tt x}+B{\tt x}^q+C{\tt x}^{2q-1}$. Hence Theorem D follows from Theorem A.

%%%%%%%%%%%%%%%%%%%%%%%%%%%%%%%%%%%%%%%%
%        section 2
%%%%%%%%%%%%%%%%%%%%%%%%%%%%%%%%%%%%%%%%

\section{Statements of Theorems A and B}

The main results of the paper are the following theorems.

\medskip
\noindent{\bf Theorem A.} {\em Let $f=a{\tt x}+b{\tt x}^q+{\tt x}^{2q-1}\in\Bbb F_q[{\tt x}]$, where $q$ is odd. Then $f$ is a PP of $\Bbb F_{q^2}$ if and only if one of the following is satisfied.

\begin{itemize}
  \item [(i)] $a(a-1)$ is a square in $\Bbb F_q^*$, and $b^2=a^2+3a$.
  \item [(ii)] $a=1$, and $b^2-4$ is a square in $\Bbb F_q^*$.
  \item [(iii)] $a=3$, $b=0$, $q\equiv -1\pmod 6$.
  \item [(iv)] $a=b=0$, $q\equiv 1,3\pmod 6$.
\end{itemize} 
}
\medskip
  
\medskip
\noindent{\bf Theorem B.} {\em Let $f=a{\tt x}+b{\tt x}^q+{\tt x}^{2q-1}\in\Bbb F_q[{\tt x}]$, where $q$ is even. Then $f$ is a PP of $\Bbb F_{q^2}$ if and only if one of the following is satisfied.

\begin{itemize}
  \item [(i)] $q>2$, $a\ne 1$, $\text{\rm Tr}_{q/2}(\frac 1{a+1})=0$, $b^2=a^2+a$.
  \item [(ii)] $q>2$, $a=1$, $b\ne 0$, $\text{\rm Tr}_{q/2}(\frac 1b)=0$.
\end{itemize} 
}
\medskip

In Theorem~B (i), we can write $\frac 1{a+1}=d^2+d^4$, where $d\in\Bbb F_q\setminus \Bbb F_2$. Then $(a,b,1)=\frac 1{d^2+d^4}(1+d^2+d^4,1+d+d^2,d^2+d^4)$. Similarly, in Theorem~B (ii), we can write $\frac 1b=d+d^2$, $d\in\Bbb F_q\setminus\Bbb F_2$. Then $(a,b,1)=\frac 1{d+d^2}(d+d^2,1,d+d^2)$. Let $\text{PG}(2,\Bbb F_q)$ denote the projective plane over $\Bbb F_q$ and define
\[
\mathcal X=\{[a:b:c]\in \text{PG}(2,\Bbb F_q): a{\tt x}+b{\tt x}^q+c{\tt x}^{2q-1}\ \text{is a PP of $\Bbb F_{q^2}$}\}.
\]
Then for even $q$ we have
\[
\begin{split}
\mathcal X=\,&\bigl\{[1+d^2+d^4:1+d+d^2:d^2+d^4]:d\in\Bbb F_q\setminus\Bbb F_2\bigr\}\cr
&\cup  \bigl\{[d+d^2:1:d+d^2]:d\in\Bbb F_q\setminus\Bbb F_2\bigr\}\cr
&\cup  \bigl\{[d:1:0]:d\in\Bbb F_q\setminus\{1\}\bigr\}\cr
&\cup  \bigl\{[1:0:0]\bigr\}.
\end{split}
\]

%%%%%%%%%%%%%%%%%%%%%%%%%%%%%%%%%%%%%%%%
%     section 3
%%%%%%%%%%%%%%%%%%%%%%%%%%%%%%%%%%%%%%%%

\section{Proof of Theorem A}

%%%%%%%%%%%%%%%%%%%%%%%%%%%%%%%%%%%%%%%%

\subsection{The case $a(a-1)b=0$}\label{s3.1}\

\medskip
Let $f=a{\tt x}+b{\tt x}^q+{\tt x}^{2q-1}\in\Bbb F_q[{\tt x}]$, where $q$ is odd. We first prove Theorem~A under the assumption $a(a-1)b=0$.

\medskip
{\bf Case 1.} Assume $a=b=0$. Then $f={\tt x}^{2q-1}$ is a PP of $\Bbb F_{q^2}$ if and only if $\text{gcd}(2q-1,q^2-1)=1$, i.e., $q\equiv 1,3\pmod 6$.

\medskip
{\bf Case 2.} Assume $a\ne 0$, $b=0$. By \cite[Theorem~1.1]{Hou1}, $f=a{\tt x}+{\tt x}^{2q-1}$ is a PP of $\Bbb F_{q^2}$ if and only if one of the following occurs:

\begin{itemize}
  \item [(a)] $a=1$, $q\equiv 1\pmod 4$;
  \item [(b)] $a=-3$, $q\equiv\pm1\pmod{12}$;
  \item [(c)] $a=3$, $q\equiv -1\pmod 6$.
\end{itemize}  
Condition (c) is (iii) in Theorem~A; condition (a) is equivalent to (ii) in Theorem~A with $b=0$. Note that $3$ is square in $\Bbb F_q^*$ if and only if $q\equiv\pm1\pmod{12}$ \cite[\S5.2]{Ire-Ros}. Hence condition (b) is equivalent to (i) in Theorem~A with $b=0$.

\medskip
{\bf Case 3.} Assume $a=0$, $b\ne 0$. For integers $\alpha,\beta\ge 0$ with $\alpha+\beta=q-1$, it follows from \eqref{3.16} that
\[
\sum_{x\in\Bbb F_{q^2}}f(x)^{\alpha+\beta q}=-\sum_{\substack{k,l\cr \alpha+1+k-l=0,\,q+1}}\binom\alpha k\binom\beta l b^{-(k+l)}.
\]
Setting $\alpha=q-1$ and $\beta=0$, we have  
\[
\sum_{x\in\Bbb F_{q^2}}f(x)^{q-1}=-\binom{q-1}1b^{-1}=b^{-1}\ne 0.
\]
By Hermite's criterion \cite[Lemma 7.3]{LN}, $f$ cannot be a PP of $\Bbb F_{q^2}$ in this case.

\medskip
{\bf Case 4.} Assume $a=1$. We show that $f={\tt x}+b{\tt x}^q+{\tt x}^{2q-1}$ is a PP of $\Bbb F_{q^2}$ if and only if $f_1={\tt x}^2+b{\tt x}+1$ has two distinct roots in $\Bbb F_q$.

($\Leftarrow$) Let $x,y\in\Bbb F_{q^2}$ such that $f(x)=y$. We show that $x$ is uniquely determined by $y$.

First assume $y\ne 0$. Let $t=xy=x^2+x^{2q}+bx^{q+1}\in\Bbb F_q$. Then 
\[
\frac ty+b\Bigl(\frac ty\Bigr)^q+\Bigl(\frac ty\Bigr)^{2q-1}=y,
\]
i.e.,
\[
t\Bigl(\frac 1y+\frac b{y^q}+\frac 1{y^{2q-1}}\Bigr)=y.
\]
Hence $t$ is unique. It follows that $x$ is unique.

Nest assume $y=0$. We claim that $x=0$. Assume to the contrary that $x\ne 0$. Then we have
\begin{equation}\label{3.1.0}
1+bx^{q-1}+x^{2(q-1)}=0,
\end{equation}
i.e.,
$f_1(x^{q-1})=0$. Thus $x^{q-1}\in\Bbb F_q$. Therefore $x^{2(1-q)}=x^{(q-1)^2}=1$, and hence $x^{q-1}=\pm1$. It follows from \eqref{3.1.0} that $b=\pm 2$. Thus $f_1=({\tt x}\pm1)^2$, which is a contradiction.

($\Rightarrow$) Assume to the contrary that $f_1$ does not have two distinct roots in $\Bbb F_q$. 

If $f_1$ is irreducible over $\Bbb F_q$, let $y\in\Bbb F_{q^2}$ be a root of $f_1$. Since $y^{1+q}=\text{N}_{q^2/q}(y)=1$, there exists an $x\in\Bbb F_{q^2}^*$ such that $y=x^{q-1}$. Then $f(x)=xf_1(y)=0=f(0)$, which is a contradiction.

If $f_1$ is not irreducible over $\Bbb F_q$, then $f_1=({\tt x}-\epsilon)^2$, where $\epsilon=1$ or $-1$. Since $\epsilon^{1+q}=1$, again there exists $x\in\Bbb F_{q^2}^*$ such that $\epsilon=x^{q-1}$. Then $f(x)=xf_1(\epsilon)=0=f(0)$, which is a contradiction.

This completes the proof of Theorem~A under the assumption $a(a-1)b=0$. In the next two subsections, we assume that $a(a-1)b\ne 0$ and we prove that $f$ is a PP of $\Bbb F_{q^2}$ if and only if $a(a-1)$ is a square in $\Bbb F_q^*$ and $b^2=a^2+3a$.

%%%%%%%%%%%%%%%%%%%%%%%%%%%%%%%%%%%%%%%%

\subsection{The case $a(a-1)b\ne 0$, sufficiency}\label{s3.2}\

\medskip
Assume that $a(a-1)$ is a square in $\Bbb F_q^*$ and $b^2=a^2+3a$.

\medskip

$1^\circ$ We claim that $f(\Bbb F_{q^2}\setminus\Bbb F_q)\subset\Bbb F_{q^2}\setminus\Bbb F_q$.

Assume to the contrary that there exists an $x\in\Bbb F_{q^2}\setminus\Bbb F_q$ such that $f(x)^q=f(x)$. Then 
\[
ax^q+bx+x^{2-q}=ax+bx^q+x^{2q-1},
\]
i.e.,
\[
(a-b)(x^q-x)+\frac{x^3-x^{3q}}{x^{1+q}}=0.
\]
Since $x^q-x\ne 0$, we have
\[
a-b-\frac{x^2+x^{1+q}+x^{2q}}{x^{1+q}}=0,
\]
i.e.,
\begin{equation}\label{3.1}
x^{2(q-1)}-(a-b-1)x^{q-1}+1=0.
\end{equation}
Using the relation $b^2=a^2+3a$, we find that
\[
(a-b-1)^2-4=\frac{a-1}a(a-b)^2,
\]
which is a square in $\Bbb F_q^*$. So ${\tt x}^2-(a-b-1){\tt x}+1$ is reducible over $\Bbb F_q$. Thus by \eqref{3.1}, we have $x^{q-1}\in\Bbb F_q$. Then $1=x^{(q-1)^2}=x^{q^2-2q+1}=x^{2(1-q)}$. Since $x\notin \Bbb F_q$, we must have $x^{1-q}=-1$. Then \eqref{3.1} becomes $a-b+1=0$. However, we have 
\begin{equation}\label{3.2}
(a+b+1)(a-b+1)=(a+1)^2-b^2=(a+1)^2-(a^2+3a)=1-a\ne 0,
\end{equation}
which is a contradiction.

\medskip
$2^\circ$ Let $x,y\in\Bbb F_{q^2}$ such that $f(x)=y$. We show that $x$ is uniquely determined by $y$. 

If $y\in\Bbb F_q$, by $1^\circ$, we have $x\in \Bbb F_q$, so $f(x)=(a+b+1)x$. By \eqref{3.2}, $a+b+1\ne 0$, so we must have $x=\frac y{a+b+1}$.

Therefore, we assume $y\in\Bbb F_{q^2}\setminus\Bbb F_q$. It follows that $x\in \Bbb F_{q^2}\setminus\Bbb F_q$.

\medskip
$3^\circ$ We write $\text{T}=\text{Tr}_{q^2/q}$ and $\text{N}={\text N}_{q^2/q}$. It suffices to show that $\text{T}(x)$ and $\text{N}(x)$ are uniquely determined by $y$. (If $x_1\in\Bbb F_{q^2}\setminus\Bbb F_q$ such that $f(x_1)=y$ and $\text{T}(x_1)=\text{T}(x)$, $\text{N}(x_1)=\text{N}(x)$, then $x_1=x$ or $x^q$. Since $f(x^q)=f(x)^q=y^q\ne y$, we must have $x_1=x$.)

We have 
\begin{equation}\label{3.3}
\begin{cases}
\text{T}(f(x))=\text{T}(y),\cr
\text{N}(f(x))=\text{N}(y),
\end{cases}
\end{equation}
where
\begin{equation}\label{3.4}
\text{T}(f(x))=(a+b)\text{T}(x)+\text{T}(x^{2q-1}),
\end{equation}
\begin{equation}\label{3.5}
\begin{split}
\text{N}(f(x))\,&=(ax+bx^q+x^{2q-1})(ax^q+bx+x^{2-q})\cr
&=(a^2+b^2+1)\text{N}(x)+(ab+b)\text{T}(x^2)+a\text{T}(x^{3-q}).
\end{split}
\end{equation}
We wish to express $\text{T}(f(x))$ and $\text{N}(f(x))$ in terms of $\text{T}(x)$ and $\text{N}(x)$. For this purpose, we need a few formulas: For $z\in\Bbb F_{q^2}^*$, we have
\begin{equation}\label{3.6}
\text{T}(z^2)=\text{T}(z)^2-2\text{N}(z),
\end{equation}
\[
\text{T}(z^3)=\text{T}(z)^3-2\text{N}(z)\text{T}(z),
\]
\begin{equation}\label{3.7}
\text{T}(z^{2q-1})=\text{T}(z^{3q}\cdot z^{-(1+q)})=\frac{\text{T}(z^3)}{\text{N}(z)}=\frac{\text{T}(z)^3}{\text{N}(z)}-3\text{T}(z),
\end{equation}
\begin{equation}\label{3.8}
\begin{split}
\text{T}(z^{3-q})\,&=\text{T}(z^4\cdot z^{-(1+q)})=\frac{\text{T}(z^4)}{\text{N}(z)}\cr
&=\frac 1{\text{N}(z)}\bigl[\text{T}(z^2)^2-2\text{N}(z^2)\bigr]\cr
&=\frac 1{\text{N}(z)}\bigl[\bigl(\text{T}(z)^2-2\text{N}(z)\bigr)^2-2\text{N}(z^2)\bigr]\cr
&=\frac {\text{T}(z)^4}{\text{N}(z)}-4\text{T}(z)^2+2\text{N}(z).
\end{split}
\end{equation}
Put $t=\text{T}(x)$, $n=\text{N}(x)$, $\tau=\text{T}(y)$, $\eta=\text{N}(y)$. By \eqref{3.4} -- \eqref{3.8}, we have 
\[
\text{T}(f(x))=(a+b)t+\frac{t^3}n-3t=\frac{t^3}n+(a+b-3)t,
\]
\[
\begin{split}
\text{N}(f(x))\,&=(a^2+b^2+1)n+(ab+b)(t^2-2n)+a\Bigl(\frac{t^4}n-4t^2+2n\Bigr)\cr
&=a\frac{t^4}n+(ab-4a+b)t^2+(a-b+1)^2n.
\end{split}
\]
Then \eqref{3.3} becomes
\begin{equation}\label{3.9}
\begin{cases}
\displaystyle \frac{t^3}n+(a+b-3)t=\tau,\vspace{2mm}\cr
\displaystyle a\frac{t^4}n+(ab-4a+b)t^2+(a-b+1)^2n=\eta.
\end{cases}
\end{equation}
We proceed to show that \eqref{3.9} has at most one solution $(t,n)\in\Bbb F_q\times \Bbb F_q$. 

First assume $\tau=0$. Since $y\in\Bbb F_{q^2}\setminus\Bbb F_q$, $q$ must be odd. We claim that $t=0$. If, to the contrary, $t\ne 0$, then by the first equation of \eqref{3.9}, we have $\frac{t^2}n=-(a+b-3)$. Using the relation $b^2=a^2+3a$, we find that 
\[
\frac{t^2}n\Bigl(\frac{t^2}n-4\Bigr)=(a+b+1)(a+b-3)=\frac{a-1}a(a+b)^2,
\]
which is a square in $\Bbb F_q$. Then $x\in\Bbb F_q$, which is a contradiction. So the claim is proved. By the second equation of \eqref{3.9}, we have $n=\frac{\eta}{(a-b+1)^2}$. Hence $(t,n)$ is unique,

Now assume $\tau\ne 0$. It follows that $t\ne 0$. Put $s=\frac{t^2}n$ and $\sigma=\frac{\tau^2}\eta$, and write \eqref{3.9} as
\begin{equation}\label{3.10}
\begin{cases}
t(s+a+b-3)=\tau,\vspace{2mm}\cr
\displaystyle t^2\Bigl(as+(ab-4a+b)+(a-b+1)^2\,\frac 1s\Bigr)=\frac{\tau^2}\sigma.
\end{cases}
\end{equation}
Eliminating $t$ and $\tau$ in \eqref{3.10}, we have
\[
\frac{as+(ab-4a+b)+(a-b+1)^2\frac 1s}{(s+a+b-3)^2}=\frac 1\sigma,
\]
i.e.,
\begin{equation}\label{3.11}
s^3+(-a\sigma+2a+2b-6)s^2+\bigl[(4a-b-ab)\sigma+(a+b-3)^2\bigr]s-(a-b+1)^2\sigma=0.
\end{equation}
It suffices to show that \eqref{3.11} has at most one solution $s\in\Bbb F_q$. Let $g({\tt s})\in\Bbb F_q[{\tt s}]$ denote the polynomial given by the left side of \eqref{3.11}. We find that the discriminant of $g$ is given by 
\begin{equation}\label{3.12}
D(g)=(a-1)^2\sigma(\sigma-4)h(\sigma),
\end{equation}
where
\begin{equation}\label{3.13}
h(\sigma)=a^2(b^2-4a)\sigma^2-2\bigl(ab(a+b)^2-8a^3-6a^2b-2b^3+9ab\bigr)\sigma+(a+b+1)(a+b-3).
\end{equation}
Here we emphasize that \eqref{3.12} and \eqref{3.13} hold with $a$ and $b$ treated as independent parameters. Using the relation $b^2=a^2+3a$, we find that 
\begin{equation}\label{3.14}
h(\sigma)=a^3(a-1)\Bigl(\sigma-\frac{2a^2+2ab-3b}{a^2}\Bigr)^2.
\end{equation}
(Note: Equations \eqref{3.12} and \eqref{3.14}, especially \eqref{3.12}, are painful to compute by hand, but they are easily obtained using a symbolic computation program.) By \eqref{3.12} and \eqref{3.14},
\begin{equation}\label{3.15}
D(g)=a^3(a-1)^3\sigma(\sigma-4)\Bigl(\sigma-\frac{2a^2+2ab-3b}{a^2}\Bigr)^2.
\end{equation}
In \eqref{3.15}, $\sigma(\sigma-4)$ is a nonsquare in $\Bbb F_q^*$ since $y\in\Bbb F_{q^2}\setminus\Bbb F_q$. If $\sigma\ne\frac{2a^2+2ab-3b}{a^2}$, then $D(g)$ is a nonsquare in $\Bbb F_q^*$. Therefore $g$ has at most one root in $\Bbb F_q$, and we are done.

\medskip
$4^\circ$ Now assume $\sigma=\frac{2a^2+2ab-3b}{a^2}$. We have
\[
\begin{split}
\sigma(\sigma-4)\,&=\Bigl(2+\frac{b(2a-3)}{a^2}\Bigr)\Bigl(-2+\frac{b(2a-3)}{a^2}\Bigr)=\frac{b^2(2a-3)^2}{a^4}-4\cr
&=\frac{a(a+3)(2a-3)^2}{a^4}-4=-\frac{27(a-1)}{a^3}.
\end{split}
\]
Since $\sigma(\sigma-4)\ne 0$ (a nonsquare in $\Bbb F_q^*$), we have $3\nmid q$. Using the relations $\sigma=\frac{2a^2+2ab-3b}{a^2}$ and $b^2=a^2+3a$, we find that
\[
g({\tt s})={\tt s}^3+\frac{3(-2a+b)}a{\tt s}^2+\frac{3(5a-4b+3)}a{\tt s}+\frac{-14a^2+13ab-18a+3b}{a^2}.
\]
Then
\[
g'=3{\tt s}^2+\frac{6(-2a+b)}a{\tt s}+\frac{3(5a-4b+3)}a.
\]
The discriminant of $g'$ is given by
\[
D(g')=\Bigl[\frac{2(-2a+b)}a\Bigr]^2-4\cdot\frac{5a-4b+3}a.
\]
Using the relation $b^2=a^2+3a$, we find that 
\[
D(g')=0.
\]
Thus we have 
\[
g'=3\Bigl({\tt s}+\frac{-2a+b}a\Bigr)^2.
\]
Since $D(g)=0$, $\text{gcd}(g,g')\ne 1$. Thus we must have 
\[
g=\Bigl({\tt s}+\frac{-2a+b}a\Bigr)^3.
\]
In particular, $g$ has a unique root in $\Bbb F_q$. This completes the proof of the sufficiency part of Theorem~A under the assumption $a(a-1)b\ne 0$. 

%%%%%%%%%%%%%%%%%%%%%%%%%%%%%%%%%%%%%%%%

\subsection{The case $a(a-1)b\ne 0$, necessity}\label{s3.3}\

\medskip

Let $f=a{\tt x}+b{\tt x}^q+{\tt x}^{2q-1}\in\Bbb F_q$. Let $0\le s<q^2-1$ and write $s=\alpha+\beta q$, where $0\le \alpha,\beta\le q-1$. One has 
\[
\begin{split}
&\sum_{x\in\Bbb F_{q^2}}f(x)^s \kern 9cm \cr
=\,&\sum_{x\in\Bbb F_{q^2}^*}(ax+bx^q+x^{2q-1})^{\alpha+\beta q}\cr
=\,&\sum_{x\in\Bbb F_{q^2}^*}(ax+bx^q+x^{2q-1})^\alpha(ax^q+bx+x^{2-q})^\beta\cr
=\,&\sum_{x\in\Bbb F_{q^2}^*}\sum_{i,j,k,l}\binom\alpha i\binom ik(ax)^{\alpha-i}(bx^q)^{i-k}(x^{2q-1})^k\binom\beta j\binom jl(ax^q)^{\beta-j}(bx)^{j-l}(x^{2-q})^l\cr
=\,&\sum_{x\in\Bbb F_{q^2}^*}\sum_{i,j,k,l}\binom\alpha i\binom ik\binom\beta j\binom jl a^{\alpha+\beta-i-j}b^{i+j-k-l}x^{\alpha+\beta q+(q-1)(i+k-j-l)}.
\end{split}
\]
If $\alpha+\beta q\not\equiv 0\pmod{q-1}$, then clearly $\sum_{x\in\Bbb F_{q^2}}f(x)^s=0$. Assume $\alpha+\beta q\equiv 0\pmod{q-1}$. Then one must have $\alpha+\beta=q-1$, and the above calculation becomes
\[
\begin{split}
\sum_{x\in\Bbb F_{q^2}}f(x)^s\,&=\sum_{x\in\Bbb F_{q^2}^*}\sum_{i,j,k,l}\binom\alpha i\binom ik\binom\beta j\binom jl a^{q-1-i-j}b^{i+j-k-l}x^{(q-1)(q-\alpha+i+k-j-l)}\cr
&=-\sum_{\substack{i,j,k,l\cr q-\alpha+i+k-j-l\equiv 0\, (\text{mod}\, q+1)}}\binom\alpha i\binom ik\binom\beta j\binom jl a^{q-1-i-j}b^{i+j-k-l}.
\end{split}
\]
For $0\le k\le i\le\alpha$ and $0\le l\le j\le\beta$, one has 
\[
-(q+1)<q-\alpha+i+k-j-l<2(q+1).
\]
Hence
\begin{equation}\label{3.16}
\sum_{x\in\Bbb F_{q^2}}f(x)^s=-\sum_{q-\alpha+i+k-j-l= 0,\, q+1}\binom\alpha i\binom ik\binom\beta j\binom jl a^{q-1-i-j}b^{i+j-k-l}.
\end{equation}

Now assume that $a(a-1)b\ne 0$ and $f$ is a PP of $\Bbb F_{q^2}$. We proceed to prove that $a(a-1)$ is a square in $\Bbb F_q^*$ and $b^2=a^2+3a$. The proof relies on several lemmas which are provided after the proof.

Letting $\alpha=0$ and $\beta=q-1$ in \eqref{3.16}, one has
\begin{equation}\label{3.16.0}
\begin{split}
0\,&=-\sum_{x\in\Bbb F_{q^2}}f(x)^{(q-1)q}=\sum_{q-j-l=0}\binom{q-1}j\binom jl a^{-j}b^{j-l}\cr
&=\sum_{1\le l\le \frac q2}\binom{q-1}{q-l}\binom{q-l}l a^{-(q-l)}b^{q-2l}\cr
&=\frac ba\sum_{1\le l\le \frac q2}(-1)^{q-l}\binom{-l}la^lb^{-2l}\cr
&=-\frac ba\sum_{1\le l\le \frac q2}\binom{-l}l\Bigl(\frac{-a}{b^2}\Bigr)^l\cr
&=-\frac ba\biggl(\sum_{0\le l\le \frac q2}\binom{-l}l\Bigl(\frac{-a}{b^2}\Bigr)^l-1\biggr).
\end{split}
\end{equation}
Thus $\sum_{0\le l\le \frac q2}\binom{-l}l(\frac{-a}{b^2})^l=1$. By Lemma~\ref{L3.1}, $1+4\cdot\frac{-a}{b^2}$ is a square in $\Bbb F_q^*$, i.e., $b^2-4a$ is a square in $\Bbb F_q^*$. (Note that $b^2-4a=a(a-1)$ after we prove that $b^2=a^2+3a$.)

Next we prove that $b^2=a^2+3a$. Letting $\alpha=1$ and $\beta=q-2$ in \eqref{3.16}, one has 
\[
\begin{split}
&-\sum_{x\in\Bbb F_{q^2}}f(x)^{1+(q-2)q}\cr
=\,&\sum_{q-1-j-l=0}\binom{q-2}j\binom jl a^{-j}b^{j-l}+\sum_{q-j-l=0}\binom{q-2}j\binom jl a^{-1-j}b^{1+j-l}\cr
&+\sum_{q+1-j-l=0,\,q+1}\binom{q-2}j\binom jl a^{-1-j}b^{j-l}\cr
=\,&\sum_{1\le l\le\frac{q-1}2}\binom{q-2}{q-1-l}\binom{q-1-l}l a^{-(q-1-l)}b^{q-1-2l}\cr
&+\sum_{2\le l\le\frac q2}\binom{q-2}{q-l}\binom{q-l}l a^{-1-(q-l)}b^{1+q-2l}\cr
&+\sum_{3\le l\le\frac{q+1}2}\binom{q-2}{q+1-l}\binom{q+1-l}l a^{-1-(q+1-l)}b^{q+1-2l}+a^{-1}.
\end{split}
\]
Since $\binom{-2}k=(-1)^k(k+1)$ for $k\ge 0$, the above calculation gives
\[
\begin{split}
&-\sum_{x\in\Bbb F_{q^2}}f(x)^{1+(q-2)q}\cr
=\,&\sum_{1\le l\le\frac{q-1}2}(-1)^{q-1-l}(q-l)\binom{-1-l}la^lb^{-2l}+\sum_{2\le l\le\frac q2}(-1)^{q-l}(q-l+1)\binom{-l}l a^{l-2}b^{-2l+2}\cr
&+\sum_{3\le l\le\frac{q+1}2}(-1)^{q+1-l}(q+2-l)\binom{1-l}la^{l-3}b^{-2l+2}+a^{-1}\cr
=\,&\sum_{1\le l\le\frac{q-1}2}(l+1)\binom{-l}{l+1}(-1)^la^lb^{-2l}+\sum_{2\le l\le\frac q2}(l-1)\binom{-l}l(-1)^la^{l-2}b^{-2l+2}\cr
&+\sum_{3\le l\le\frac{q+1}2}(l-2)\binom{-(l-1)}l(-1)^{l+1}a^{l-3}b^{-2l+2}+a^{-1}.
\end{split}
\]
Put $z=\frac{-a}{b^2}$. Then one has 
\begin{equation}\label{3.17}
\begin{split}
&-\sum_{x\in\Bbb F_{q^2}}f(x)^{1+(q-2)q}\cr
=\,&\sum_{0\le l\le\frac{q-1}2}(l+1)\binom{-l}{l+1}z^l+\frac{b^2}{a^2}\sum_{0\le l\le\frac q2}(l-1)\binom{-l}lz^l+\frac{b^2}{a^2}\cr
&+\sum_{2\le l\le\frac{q-1}2}(l-1)\binom{-l}{l+1}(-1)^la^{l-2}b^{-2l}+a^{-1}\cr
=\,&\sum_{0\le l\le\frac{q-1}2}\bigl[l+1+(l-1)a^{-2}\bigr]\binom{-l}{l+1}z^l+\frac{b^2}{a^2}\sum_{0\le l\le\frac q2}(l-1)\binom{-l}l(-1)^lz^l+\frac{b^2}{a^2}+\frac 1a\cr
\end{split}
\end{equation}
\[
\begin{split}
=\,&(1+a^{-2})\sum_{0\le l\le\frac{q-1}2}(l+1)\binom{-l}{l+1}z^l-2a^{-2}\sum_{0\le l\le\frac{q-1}2}\binom{-l}{l+1}z^l\cr
&+\frac{b^2}{a^2}\sum_{0\le l\le\frac q2}(l+1)\binom{-l}lz^l-2\,\frac{b^2}{a^2}\sum_{0\le l\le\frac q2}\binom{-l}lz^l+\frac{b^2}{a^2}+\frac 1a\cr
=\,&(1+a^{-2})\frac{2z}{1+4z}-2a^{-2}+\frac{b^2}{a^2}\cdot\frac{1+3z}{1+4z}-2\,\frac{b^2}{a^2}+\frac{b^2}{a^2}+\frac 1a \kern 5mm\text{(by Lemmas~\ref{L3.1} -- \ref{L3.2})}\cr
=\,&\frac{2(a-1)(b^2-a^2-3a)}{a^2(b^2-4a)}.
\end{split}
\]
Since $f$ is a PP of $\Bbb F_{q^2}$, one has $\sum_{x\in\Bbb F_{q^2}}f(x)^{1+(q-2)q}=0$. Hence $b^2-a^2-3a=0$. This completes the proof of the necessity part of Theorem~A under the assumption $a(a-1)b\ne 0$.

\medskip
The following lemmas, used in the above proof, hold for all (odd and even) $q$.

\begin{lem}\label{L3.1}\cite[Lemma 5.1]{Hou2} Let $z\in\Bbb F_q^*$ and write ${\tt x}^2+{\tt x}-z=({\tt x}-r_1)({\tt x}-r_2)$, $r_1,r_2\in\Bbb F_{q^2}$. Then 
\[
\sum_{0\le l\le \frac q2}\binom{-l}l z^l=
\begin{cases}
\displaystyle \frac 12&\text{if}\ r_1=r_2\in\Bbb F_q,\vspace{1mm}\cr
1& \text{if}\ r_1,r_2\in\Bbb F_q,\ r_1\ne r_2,\cr
0&\text{if}\ r_1,r_2\notin\Bbb F_q.
\end{cases}
\]
\end{lem}

\begin{lem}\label{L3.3} Let $z\in\Bbb F_q^*$ such that ${\tt x}^2+{\tt x}-z$ has two distinct roots in $\Bbb F_q$. Then 
\[
\sum_{0\le l\le \frac q2}(l+1)\binom{-l}l z^l=\frac{1+3z}{1+4z}.
\]
\end{lem}

\begin{proof}
We denote the constant term of a Laurent series in ${\tt x}$ by $\text{ct}(\ )$. We have
\[
\begin{split}
\sum_{0\le l\le \frac q2}(l+1)\binom{-l}l z^l\,&=\sum_{0\le l\le q-2}(l+1)\binom{-l}l z^l\cr
&=\sum_{0\le l\le q-2}(l+1)\cdot\text{ct}\Bigl(\frac 1{{\tt x}^l(1+{\tt x})^l}\Bigr)\cdot z^l\cr
&=\text{ct}\biggl[\,\sum_{0\le l\le q-2}(l+1)\Bigl(\frac z{{\tt x}^l(1+{\tt x})}\Bigr)^l \, \biggr].
\end{split}
\]
Since 
\[
\sum_{1\le l\le q-1}l{\tt y}^{l-1}=\frac d{d{\tt y}}\Bigl(\frac{1-{\tt y}^q}{1-{\tt y}}\Bigr)=\frac{1-{\tt y}^q}{(1-{\tt y})^2},
\]
we have
\[
\sum_{0\le l\le q-2}(l+1)\Bigl(\frac z{{\tt x}(1+{\tt x})}\Bigr)^l=\frac{1-(\frac z{{\tt x}(1+{\tt x})})^q}{(1-\frac z{{\tt x}(1+{\tt x})})^2}
=\Bigl(\frac {{\tt x}(1+{\tt x})}{{\tt x}(1+{\tt x})-z}\Bigr)^2\Bigl(1-\frac z{{\tt x}^q}+\frac z{(1+{\tt x})^q}\Bigr).
\]
Thus
\[
\begin{split}
\sum_{0\le l\le \frac q2}(l+1)\binom{-l}l z^l\,&=\text{ct}\biggl[-\frac z{{\tt x}^q}\Bigl(\frac {{\tt x}(1+{\tt x})}{{\tt x}(1+{\tt x})-z}\Bigr)^2\biggr]\cr
&=\text{ct}\biggl[-\frac z{{\tt x}^q}\Bigl(1+\frac z{{\tt x}^2+{\tt x}-z}\Bigr)^2\biggr].
\end{split}
\]
The rest of the calculation is almost identical to that in the proof of \cite[Lemma~5.3]{Hou2}. We omit the details.
\end{proof}

\begin{lem}\cite[Lemmas~5.2 and 5.3]{Hou2}\label{L3.2} Let $z\in\Bbb F_q^*$ such that ${\tt x}^2+{\tt x}-z$ has two distinct roots in $\Bbb F_q$. Then 
\[
\sum_{0\le l\le \frac{q-1}2}\binom{-l}{l+1} z^l=1,
\]
\[
\sum_{0\le l\le \frac{q-1}2}(l+1)\binom{-l}{l+1} z^l=\frac{2z}{1+4z}.
\]
\end{lem}

%%%%%%%%%%%%%%%%%%%%%%%%%%%%%%%%%%%%%%%%
%    section 4
%%%%%%%%%%%%%%%%%%%%%%%%%%%%%%%%%%%%%%%%

\section{Proof of Theorem~B}

We follow the same outline of the proof of Theorem~A. However, certain critical arguments in the proof of Theorem~A fail in characteristic $2$, and they have to be replaced with new approaches. First, in Subsection~\ref{s3.2}, the discriminant $D(g)$ in \eqref{3.12}, which was at the heart of the proof there, is rendered useless in characteristic $2$. Second, in Subsection~\ref{s3.3}, the calculation in \eqref{3.17} does not produce any useful information, again because of the even characteristic.

%%%%%%%%%%%%%%%%%%%%%%%%%%%%%%%%%%%%%%%%
\subsection{The case $a(a-1)b=0$}\

\medskip

Let $f=a{\tt x}+b{\tt x}^q+{\tt x}^{2q-1}\in\Bbb F_q[{\tt x}]$, where $q$ is even. We first prove Theorem~B under the assumption $a(a-1)b=0$.

\medskip
{\bf Case 1.} Assume $a=b=0$. Then $f={\tt x}^{2q-1}$ is a PP of $\Bbb F_{q^2}$ if and only if $\text{gcd}(2q-1,q^2-1)=1$, i.e., $q=2^{2k}$, which is equivalent to (i) in Theorem~B with $a=b=0$.

\medskip
{\bf Case 2.} Assume $a\ne 0$, $b=0$. By \cite[Theorem~1.1]{Hou1}, $f=a{\tt x}+{\tt x}^{2q-1}$ is never a PP of $\Bbb F_{q^2}$.
  
\medskip
{\bf Case 3.} Assume $a=0$, $b\ne 0$. By Case 3 in Subsection~\ref{s3.1}, $f$ cannot be a PP of $\Bbb F_{q^2}$.

\medskip
{\bf Case 4.} Assume $a=1$. The conclusion in Case 4 of Subsection~\ref{s3.1} also holds for characteristic $2$: $f$ is a PP of $\Bbb F_{q^2}$ if and only if ${\tt x}^2+b{\tt x}+1$ has two distinct roots in $\Bbb F_q$, i.e., $b\ne 0$ and $\text{Tr}_{q/2}(\frac 1b)=0$, which is (ii) in Theorem~B. (Note: Theorem~B with $a=1$ also appeared as \cite[Theorem~1.2]{Hou2}.)  

%%%%%%%%%%%%%%%%%%%%%%%%%%%%%%%%%%%%%%%%
\subsection{The case $a(a-1)b\ne 0$, sufficiency}\

\medskip
We are given that $q$ ($>2$) is even, $a\in\Bbb F_q\setminus\Bbb F_2$, $\text{Tr}_{q/2}(\frac 1{a+1})=0$, and $b^2=a^2+a$. The goal is to show that $f$ is a PP of $\Bbb F_{q^2}$.

For each $y\in\Bbb F_{q^2}$, we show that there is at most one $x\in\Bbb F_{q^2}$ such that $f(x)=y$. We only have to consider the case $y\in\Bbb F_{q^2}\setminus\Bbb F_q$. (We have $\text{Tr}_{q/2}(\frac1{(a+b+1)^2})=\text{Tr}_{q/2}(\frac 1{a+1})=0$. Thus ${\tt x}^2+(a+b+1){\tt x}+1$ has two distinct roots in $\Bbb F_q$. By the argument in Subsection~\ref{s3.2}, $1^\circ$, $f(\Bbb F_{q^2}\setminus\Bbb F_q)\subset\Bbb F_{q^2}\setminus\Bbb F_q$. By Subsection~\ref{s3.2}, $2^\circ$, if $y\in\Bbb F_q$, there is precisely one $x\in\Bbb F_{q^2}$ such that $f(x)=y$.)

Put $\tau=\text{Tr}_{q^2/q}(y)$, $\eta=\text{N}_{q^2/q}(y)$, $\sigma=\frac{\tau^2}\eta$. Then $\text{Tr}_{q/2}(\frac 1\sigma)=1$ since $y\in\Bbb F_{q^2}\setminus\Bbb F_q$. By \eqref{3.11} and the argument of Subsection~\ref{s3.2}, $3^\circ$, it suffices to show that the equation 
\begin{equation}\label{4.1}
s^3+a\sigma s^2+(a+1)(b\sigma+1)s+(a+1)\sigma=0
\end{equation}
has at most one solution $s\in\Bbb F_q$. Let $u=\frac 1\sigma$ and rewrite \eqref{4.1} as
\[
\frac 1{s^3}+(b+u)\frac 1{s^2}+\frac a{a+1}\frac 1s+\frac u{a+1}=0.
\]
So it suffices to show that
\[
g:={\tt x}^3+(b+u){\tt x}^2+\frac a{a+1}{\tt x}+\frac u{a+1}
\]
has at most one root in $\Bbb F_q$. Put $A=b+u$, $B=\frac a{a+1}$, $C=\frac u{a+1}$. Assume to the contrary that $g$ has at least two distinct roots in $\Bbb F_q$. Then $g$ splits in $\Bbb F_q$. By
a theorem of K. Conrad, stated as Theorem~\ref{T4.1} at the end of this subsection, we conclude that
\[
{\tt x}^2+(AB+C){\tt x}+(A^3C+B^3+C^2)
\]
is reducible over $\Bbb F_q$. First assume $AB+C\ne 0$. Then
\begin{equation}\label{4.2}
\text{Tr}_{q/2}\Bigl(\frac{A^3C+B^3+C^2}{(AB+C)^2}\Bigr)=0.
\end{equation}
We have
\begin{equation}\label{4.2.1}
AB+C=\frac{(a+1)u+ab}{a+1},
\end{equation}
\[
A^3C+B^3+C^2=\frac 1{(a+1)^3}\bigl[(a+1)^2(b+u)^3u+a^3+(a+1)u^2\bigr].
\]
Hence
\[
\frac{A^3C+B^3+C^2}{(AB+C)^2}=\frac{(a+1)^2(b+u)^3u+a^3+(a+1)u^2}{(a+1)\bigl[(a+1)u+ab\bigr]^2}.
\]
Using the relation $b^2=a^2+a$ in the above equation, we find that  
\begin{equation}\label{4.3}
\frac{A^3C+B^3+C^2}{(AB+C)^2}=u^2+\frac 1{a+1}+\frac{bu}{a+1}+\Bigl( \frac{bu}{a+1}\Bigr)^2+\frac{a^2}{(a+1)(bu+a^2)}+\Bigl[\frac{a^2}{(a+1)(bu+a^2)} \Bigr]^2.
\end{equation}
It follows from \eqref{4.3} that 
\[
\text{Tr}_{q/2}\Bigl(\frac{A^3C+B^3+C^2}{(AB+C)^2}\Bigr)=\text{Tr}_{q/2}(u^2)=1,
\]
which contradicts \eqref{4.2}. It took some effort to find the desirable expression in \eqref{4.3}. But the verification of \eqref{4.3} should be straightforward. 

Now assume $AB+C=0$. By \eqref{4.2.1}, $u=\frac{ab}{a+1}$. Using this and the relation $b^2=a^2+a$, we see that $B^3=C^2$. Thus
\[
g={\tt x}^3+A{\tt x}^2+B{\tt x}+C={\tt x}^3+\frac CB{\tt x}^2+B{\tt x}+C=\Bigl({\tt x}+\frac CB\Bigr)({\tt x}^2+B)=\Bigl({\tt x}+\frac CB\Bigr)^3,
\]
which is again a contradiction. 

This completes the proof of the sufficiency part of Theorem~B under the assumption $a(a-1)b\ne 0$.  

\begin{thm}\label{T4.1}\cite[Thoerem 2.1]{Con}
Let $K$ be any field and $f={\tt x}^3+A{\tt x}^2+B{\tt x}+C\in K[{\tt x}]$ have roots $r_1,r_2,r_3$ in a splitting field. Then
\begin{equation}\label{4.4}
\begin{split}
&\bigl({\tt x}-(r_1^2r_2+r_2^2r_3+r_3^2r_1)\bigr)\bigl({\tt x}-(r_2^2r_1+r_1^2r_3+r_3^2r_2)\bigr)\cr
=\,&{\tt x}^2+(AB-3C){\tt x}+(A^3C+B^3+9C^2-6ABC),
\end{split}
\end{equation}
and the above quadratic polynomial has the same discriminant as $f$. 
\end{thm}

Theorem~\ref{4.1}, proved by direct computation, was used in \cite{Con} to obtain a criterion that determines whether the Galois group of a separable irreducible cubic polynomial $f$ over $K$ (of any characteristic) is $S_3$ or $A_3$: the Galois group is $S_3$ ($A_3$) if the quadratic polynomial in \eqref{4.4} is irreducible (reducible) over $K$. 

%%%%%%%%%%%%%%%%%%%%%%%%%%%%%%%%%%%%%%%%

\subsection{The case $a(a-1)b\ne 0$, necessity}\label{s4.3}\

\medskip

Assume that $q$ is even and $f=a{\tt x}+b{\tt x}^q+{\tt x}^{2q-1}\in\Bbb F_q[{\tt x}]$ is a PP of $\Bbb F_{q^2}$, where $a(a-1)b\ne 0$. The goal is to prove that $\text{Tr}_{q/2}(\frac 1{a+1})=0$ and $b^2=a^2+a$.

Let $z=\frac a{b^2}$. By \eqref{3.16.0}, $\sum_{0\le l\le\frac q2}\binom{-l}lz^l=1$. It follows from Lemma~\ref{L3.1} that ${\tt x}^2+{\tt x}+z$ is reducible over $\Bbb F_q$. Hence $\text{Tr}_{q/2}(z)=0$. (Note that $z=\frac a{b^2}=\frac 1{a+1}$ after we prove that $b^2=a^2+a$.)

It remains to show that $b^2=a^2+a$. Since $z\ne 0$ and $\text{Tr}_{q/2}(z)=0$, we must have $q>2$. Letting $\alpha=2$ and $\beta=q-3$ in \eqref{3.16}, we have 
\[
\begin{split}
&\sum_{x\in\Bbb F_{q^2}}f(x)^{2+(q-3)q}\cr
=\,&\sum_{q-2+i+k-j-l=0,\,q+1}\binom 2i\binom ik\binom{q-3}j\binom jl a^{-i-j}b^{i+j-k-l}\cr
=\,&\sum_{q-2-j-l=0,\,q+1}\binom{q-3}j\binom jl a^{-j}b^{j-l}+\sum_{q-j-l=0,\,q+1}\binom{q-3}j\binom jl a^{-2-j}b^{2+j-l}\cr
&+\sum_{q+2-j-l=0,\,q+1}\binom{q-3}j\binom jl a^{-2-j}b^{j-l}\cr
=\,&\sum_{q-2-j-l=0}\binom{q-3}j\binom jl a^{-j}b^{j-l}+\sum_{q-j-l=0}\binom{q-3}j\binom jl a^{-2-j}b^{2+j-l}\cr
&+\sum_{q+2-j-l=0}\binom{q-3}j\binom jl a^{-2-j}b^{j-l}+\sum_{j+l=1}\binom{q-3}j\binom jl a^{-2-j}b^{j-l} \cr
\end{split}
\]
\[
\begin{split}
=\,&\sum_{1\le l\le \frac q2-1}\binom{q-3}{q-2-l}\binom{q-2-l}l a^{-(q-2-l)}b^{q-2-2l}\cr
&+\sum_{3\le l\le \frac q2}\binom{q-3}{q-l}\binom{q-l}l a^{-2-(q-l)}b^{2+q-2l}\cr
&+\sum_{5\le l\le \frac q2+1}\binom{q-3}{q+2-l}\binom{q+2-l}l a^{-2-(q+2-l)}b^{q+2-2l}+a^{-3}b.
\end{split}
\]
Since $\binom{-3}k=(-1)^k\binom{k+2}2$ for all integers $k\ge 0$, the above computation continues as
\[
\begin{split}
&\sum_{x\in\Bbb F_{q^2}}f(x)^{2+(q-3)q}\cr
=\,&\sum_{1\le l\le \frac q2-1}\binom{q-l}2\binom{-2-l}l a^{1+l}b^{-1-2l}+\sum_{3\le l\le \frac q2}\binom{q-l+2}2\binom{-l}l a^{-3+l}b^{3-2l}\cr
&+\sum_{5\le l\le \frac q2+1}\binom{q+4-l}2\binom{2-l}l a^{-5+l}b^{3-2l}+a^{-3}b\cr
=\,&\frac ab\sum_{1\le l\le \frac q2-1}\binom{-l}2\binom{-l-2}l z^l+\frac{b^3}{a^3}\sum_{3\le l\le \frac q2}\binom{-l+2}2\binom{-l}l z^l\cr
&+\frac{b^3}{a^5}\sum_{5\le l\le \frac q2+1}\binom{-l+4}2\binom{-l+2}l z^l+\frac b{a^3}.
\end{split}
\]
Note that in the first sum in the above,
\[
\binom{-l-2}l=-\binom{-l-1}{l+1},
\]
and in the third sum,
\[
\begin{split}
\binom{-l+4}2\binom{-l+2}l\,&=\frac{(l-4)(l-3)}2\binom{-l+2}l\cr
&\equiv\frac{(l-3)l}2\binom{-l+2}l\pmod 2\cr
&=\frac{(l-3)(-l+2)}2\binom{-l+1}{l-1}.
\end{split}
\]
Therefore we have 
\[
\begin{split}
&\sum_{x\in\Bbb F_{q^2}}f(x)^{2+(q-3)q}\cr
=\,&\frac ab\sum_{1\le l\le \frac q2-1}\binom{-l}2\binom{-l-1}{l+1} z^l+\frac{b^3}{a^3}\sum_{3\le l\le \frac q2}\binom{-l+2}2\binom{-l}l z^l\cr
&+\frac{b^3}{a^5}\sum_{5\le l\le \frac q2+1}\frac{(l-3)(l-2)}2\binom{-l+1}{l-1} z^l+\frac b{a^3}\cr
=\,&\frac ab\sum_{0\le l\le \frac q2}\binom{-l+1}2\binom{-l}l z^{l-1}+\frac{b^3}{a^3}\biggl[\sum_{0\le l\le \frac q2}\binom{-l+2}2\binom{-l}l z^l+1\biggr]\cr
&+\frac{b^3}{a^5}\biggl[\sum_{0\le l\le \frac q2}\frac{(l-2)(l-1)}2\binom{-l}l z^{l+1}+z\biggr]+\frac b{a^3}\cr
=\,&\sum_{0\le l\le \frac q2}\binom{-l}l z^l\Bigl[b\binom{-l+1}2+\frac{b^3}{a^3}\binom{-l+2}2+\frac b{a^4}\cdot\frac{(l-2)(l-1)}2\Bigr]+\frac{b^3}{a^3}+\frac b{a^4}+\frac b{a^3}\cr
\end{split}
\]
\[
\begin{split}
=\,&\sum_{0\le l\le \frac q2}\binom{-l}l z^l\Bigl[\Bigl(b+\frac{b^3}{a^3}+\frac b{a^4}\Bigr)\binom{l+2}2+\Bigl(\frac{b^3}{a^3}+\frac b{a^4}\Bigr)(l+1)+\Bigl(b+\frac{b^3}{a^3}+\frac b{a^4}\Bigr)\Bigr]\cr
&+\frac{b^3}{a^3}+\frac b{a^4}+\frac b{a^3}.
\end{split}
\]
Using the formulas
\[
\kern-1.4cm\sum_{0\le l\le\frac q2}\binom{-l}lz^l=1\kern 3cm\text{(Lemma~\ref{L3.1})},
\]
\[
\kern-1.4cm\sum_{0\le l\le\frac q2}(l+1)\binom{-l}lz^l=1+z\kern 1.5cm\text{(Lemma~\ref{L3.3})},
\]
\begin{equation}\label{4.5}
\sum_{0\le l\le\frac q2}\binom{l+2}2\binom{-l}lz^l=1+z^2\kern 1cm\text{(Lemma~\ref{L4.2}, to be proved)},
\end{equation}
we have
\[
\begin{split}
&\sum_{x\in\Bbb F_{q^2}}f(x)^{2+(q-3)q}\cr
=\,&\Bigl(b+\frac{b^3}{a^3}+\frac b{a^4}\Bigr)(1+z^2)+\Bigl(\frac{b^3}{a^3}+\frac b{a^4}\Bigr)(1+z)+b+\frac{b^3}{a^3}+\frac b{a^4}+\frac{b^3}{a^3}+\frac b{a^4}+\frac b{a^3}\cr
=\,&\Bigl(b+\frac{b^3}{a^3}+\frac b{a^4}\Bigr)\frac{a^2}{b^4}+\Bigl(\frac{a^3}{b^3}+\frac b{a^4}\Bigr)\frac a{b^2}+\frac b{a^3}\cr
=\,&\frac{(a+1)(a+b+1)^2(a^2+b^2+a)}{a^3b^3}.
\end{split}
\]
Since $f$ is a PP of $\Bbb F_{q^2}$, the above expression equals $0$. Since $a+b+1\ne 0$ ($f(x)=(a+b+1)x$ for all $x\in\Bbb F_q$), we must have $b^2=a^2+a$. 

To complete the proof of the necessity part of Theorem~B under the assumption $a(a-1)b\ne 0$, we only need to establish \eqref{4.5}. The following lemma, which gives \eqref{4.5}, holds in all characteristics.

\begin{lem}\label{L4.2}
Let $\Bbb F_q$ be any finite field. Let $z\in\Bbb F_q^*$ such that ${\tt x}^2+{\tt x}-z$ has two distinct roots in $\Bbb F_q$. Then
\begin{equation}\label{4.5.1}
\sum_{0\le l\le \frac q2}\binom{l+2}2\binom{-l}lz^l=
\begin{cases}
1+z&\text{if}\ q=2,\vspace{2mm}\cr
\displaystyle\frac{1+6z+11z^2}{(1+4z)^2}&\text{if}\ q>2.
\end{cases}
\end{equation}
\end{lem} 

\begin{proof}
When $q=2$, \eqref{4.5.1} is easily verified. Assume $q>2$. Recall that $\text{ct}(\ )$ denotes the constant term of a Laurent series in ${\tt x}$. We have
\[
\begin{split}
S:\,&= \sum_{0\le l\le\frac q2}\binom{l+2}2\binom{-l}lz^l\cr
&=\sum_{0\le l\le q-3}\binom{l+2}2\binom{-l}lz^l\cr
&=\sum_{0\le l\le q-3}\binom{l+2}2\cdot\text{ct}\Bigl(\frac1{{\tt x}^l(1+{\tt x})^l}\Bigr)\cdot z^{l}\cr
&=\text{ct}\biggl[\,\sum_{0\le l\le q-3}\binom{l+2}2\Bigl(\frac z{{\tt x}(1+{\tt x})}\Bigr)^l\,\biggr].
\end{split}
\]
Let $\partial^2$ denote the second order Hasse derivative with respect to ${\tt y}$ \cite{Has36}. Since
\[
\sum_{2\le l\le q-1}\binom l2{\tt y}^{l-2}=\partial^2\sum_{0\le l\le q-1}{\tt y}^k=\partial^2\bigl[(1-{\tt y}^q)(1-{\tt y})^{-1}\bigr]=\frac{1-{\tt y}^q}{(1-{\tt y})^3},
\]
we have
\[
\begin{split}
\sum_{0\le l\le q-3}\binom{l+2}2\Bigl(\frac z{{\tt x}(1+{\tt x})}\Bigr)^l \,&=\frac{1-(\frac z{{\tt x}(1+{\tt x})})^q}{(1-\frac z{{\tt x}(1+{\tt x})})^3}\cr
&=\Bigl(\frac{{\tt x}(1+{\tt x})}{{\tt x}^2+{\tt x}-z}\Bigr)^3\Bigl(1-\frac z{{\tt x}^q}+\frac z{(1+{\tt x})^q}\Bigr).
\end{split}
\]
Thus
\begin{equation}\label{4.6}
\begin{split}
S\,&=\text{ct}\Bigl[\Bigl(\frac{{\tt x}(1+{\tt x})}{{\tt x}^2+{\tt x}-z}\Bigr)^3\Bigl(1-\frac z{{\tt x}^q}+\frac z{(1+{\tt x})^q}\Bigr)\Bigr]\cr
&=\text{ct}\Bigl[-\frac z{{\tt x}^q}\Bigl(\frac{{\tt x}(1+{\tt x})}{{\tt x}^2+{\tt x}-z}\Bigr)^3\Bigr]\cr
&=-z\cdot \text{ct}\Bigl[{\tt x}^{-q}\Bigl(1+\frac z{{\tt x}^2+{\tt x}-z}\Bigr)^3\Bigr].
\end{split}
\end{equation}
Write ${\tt x}^2+{\tt x}-z=({\tt x}-r_1)({\tt x}-r_2)$. Using the substitution 
\[
\frac 1{ ({\tt x}-r_1)({\tt x}-r_2)}=\frac 1{r_1-r_2}\Bigl(\frac 1{{\tt x}-r_1}-\frac 1{{\tt x}-r_2}\Bigr)
\]
repeatedly, we find that 
\begin{equation}\label{4.7}
\begin{split}
&\Bigl(1+\frac z{({\tt x}-r_1)({\tt x}-r_2)}\Bigr)^3\cr
=\,&1+(3zc-6z^2c^3+6z^3c^5)\Bigl(\frac 1{{\tt x}-r_1}-\frac 1{{\tt x}-r_2}\Bigr)\cr
&+(3z^2c^2-3z^3c^4)\Bigl(\frac 1{({\tt x}-r_1)^2}-\frac 1{({\tt x}-r_2)^2}\Bigr)+z^3c^3\Bigl(\frac 1{({\tt x}-r_1)^3}-\frac 1{({\tt x}-r_2)^3}\Bigr),
\end{split}
\end{equation}
where $c=\frac 1{r_1-r_2}$. Note that for $r\ne 0$ (in any field) and integer $k$, we have
\begin{equation}\label{4.8}
({\tt x}-r)^k=(-r)^k\Bigl(1-\frac{\tt x}r\Bigr)^k=(-r)^k\sum_{l\ge 0}\binom kl(-r)^{-l}{\tt x}^l.
\end{equation}
Combining \eqref{4.6} -- \eqref{4.8} gives 
\begin{equation}\label{4.9}
\begin{split}
S=\,&-z\biggl[(3zc-6z^2c^3+6z^3c^5)\bigl((-r_1)^{-1-q}-(-r_2)^{-1-q}\bigr)\binom{-1}q\cr
&+(3z^2c^2-3z^3c^4\bigl((-r_1)^{-2-q}+(-r_2)^{-2-q}\bigr)\binom{-2}q\cr
&+z^3c^3\bigl((-r_1)^{-3-q}-(-r_2)^{-3-q}\bigr)\binom{-3}q\biggr]\cr
=\,&-z\bigl[(3z-6z^2c^2+6z^3c^4)\cdot c\cdot (-r_1^{-2}+r_2^{-2})\cr
&+(3z^2c^2-3z^3c^4)(r_1^{-3}+r_2^{-3})+z^3c^2\cdot c\cdot(-r_1^{-4}+r_2^{-4})\bigr].
\end{split}
\end{equation}
In the above
\[
c^2=\frac 1{(r_1-r_2)^2}=\frac 1{(r_1+r_2)^2-4r_1r_2}=\frac 1{1+4z},
\]
\[
c(-r_1^{-2}+r_2^{-2})=\frac 1{r_1-r_2}\cdot\frac{r_1^2-r_2^2}{(r_1r_2)^2}=\frac{r_1+r_2}{z^2}=-\frac 1{z^2},
\]
\[
r_1^{-3}+r_2^{-3}=\frac{r_1^3+r_2^3}{(r_1r_2)^3}=\frac{(r_1+r_2)^3-3r_2r_2(r_1+r_2)}{-z^3}=\frac{-1-3z}{-z^3}=\frac{1+3z}{z^3},
\]
\[
c(-r_1^{-4}+r_2^{-4})=\frac 1{r_1-r_2}\cdot\frac{r_1^4-r_2^4}{(r_1r_2)^4}=\frac{(r_1+r_2)(r_1^2+r_2^2)}{z^4}=\frac{-(1+2z)}{z^4}.
\]
Making the above substitutions in \eqref{4.9}, we have
\[
S=\frac{1+6z+11z^2}{(1+4z)^2}.
\]
\end{proof}

%%%%%%%%%%%%%%%%%%%%%%%%%%%%%%%%%%%%%%%%
%    section 5
%%%%%%%%%%%%%%%%%%%%%%%%%%%%%%%%%%%%%%%%

\section{The Polynomial $g_{n,q}$}\label{s5}

The trinomial $a{\tt x}+b{\tt x}^q+{\tt x}^{2q-1}$ owes its origin to a class of seemingly unrelated polynomials. 

It is known that \cite{Car35,HHM}
\begin{equation}\label{5.1}
\sum_{c\in\Bbb F_q}\frac 1{{\tt x}+c}=\frac 1{{\tt x}-{\tt x}^q}.
\end{equation}
We have
\[
\begin{split}
\sum_{n\ge 0}\sum_{c\in\Bbb F_q}({\tt x}+c)^n{\tt t}^n\,&=\sum_{c\in\Bbb F_q}\frac 1{1-({\tt x}+c){\tt t}}\cr
&=\frac 1{\tt t}\sum_{c\in\Bbb F_q}\frac 1{\frac 1{\tt t}-{\tt x}-c}\cr 
&=\frac 1{\tt t}\frac 1{(\frac 1{\tt t}-{\tt x})-(\frac 1{\tt t}-{\tt x})^q}\kern 1cm \text{(by \eqref{5.1})}\cr
&=\frac {-{\tt t}^{q-1}}{1-{\tt t}^{q-1}-({\tt x}^q-{\tt x}){\tt t}^q}\cr
&=\sum_{n\ge 0} g_{n,q}({\tt x}^q-{\tt x}){\tt t}^n,
\end{split}
\]
where $g_{n,q}\in\Bbb F_p[{\tt x}]$ ($p=\text{char}\,\Bbb F_q$) is the polynomial defined by 
\[
\frac {-{\tt t}^{q-1}}{1-{\tt t}^{q-1}-{\tt x}{\tt t}^q}=\sum_{n\ge 0}g_{n,q}{\tt t}^n.
\]
Thus
\[
\sum_{c\in\Bbb F_q}({\tt x}+c)^n=g_{n,q}({\tt x}^q-{\tt x}),
\]
which can also be viewed as the definition of the polynomial $g_{n,q}$. Recent studies show that the class $g_{n,q}$ contains many new and interesting PPs \cite{FHL,Hou11,Hou12,Hou1,Hou2}. When $g_{n,q}$ is a PP of $\Bbb F_{q^e}$, we call the triple $(n,e;q)$ {\em desirable}. All desirable triples with $e=1$ are known \cite[Theorem~2.1]{FHL}. The complete determination of all desirable triples is a challenging unsolved problem. 
One of the mysterious phenomena observed in the study of the polynomial $g_{n,q}$ is that among the known desirable triples $(n,e;q)$, $n$ frequently appears in the form $q^\alpha-q^\beta-1$. Here is a summary of the previous state of knowledge of the desirable triples $(n,e;q)$ with $n=q^\alpha-q^\beta-1$.

Assume that $e\ge 2$, $n>0$, and $n\equiv q^\alpha-q^\beta-1\pmod{q^{pe}-1}$, where $0\le \alpha,\beta<pe$. (By \cite[Proposition~2.4]{Hou12}, it suffices to consider $n$ modulo $q^{pe}-1$, hence it suffices to consider $0\le\alpha,\beta<pe$.)

\begin{enumerate}
  \item If $\alpha<\beta$, then $(n,e;q)$ is desirable if and only if $(n',e;q)$ is desirable, where $n'=q^{\alpha'}-q^{\beta'}-1$, $\alpha'=pe-\alpha-\beta$, $\beta'=pe-\beta$. (So we may assume $\beta\le \alpha$.) \cite[\S5]{FHL}
  \item If $\beta=\alpha$, then $(n,e;q)$ is desirable if and only if $q>2$. \cite[\S5]{FHL}
  \item\label{itm3} If $0=\beta<\alpha$ and $q$ is even, then $(n,e;q)$ is desirable if and only if $\alpha=3$ and $q=2$, or $\alpha=2$ and $\text{gcd}(q-2,q^e-1)=1$. \cite[\S5]{FHL}
  \item\label{itm4}  Assume $0=\beta<\alpha$ and $q$ is odd. If $\alpha\le 2$, $(n,e;q)$ is desirable if and only if $\alpha=2$ and $\text{gcd}(q-2,q^e-1)=1$; if $\alpha>2$, it was conjectured that $(n,e;q)$ is not desirable. \cite[\S5]{FHL}
  \item\label{itm5} If $(\beta,\alpha)=(1,2)$, then $(n,e;q)$ is desirable if and only if $\text{gcd}(q-2,q^e-1)=1$. \cite[Corollary~5.2]{FHL}
  \item\label{itm6} If $0<\beta<\alpha$ and $\alpha\equiv\beta\equiv 0\pmod e$, then $(n,e;q)$ is desirable. \cite[Theorem~5.3]{FHL}
  \item\label{itm7} Assume $e\ge 3$ and $0<\beta<\alpha$. It was conjectured that the only desirable triples are those in \eqref{itm5} and \eqref{itm6}. \cite[Conjecture~5.5]{FHL} 
  \item\label{itm8}  Assume that $e=2$, $0<\beta<\alpha$, and $\beta$ is even. Then $(n,2;q)$ is desirable if and only if $\alpha$ is even. \cite[Remark~5.4]{FHL}
  \item\label{itm9} Assume that $e=2$, $q$ is odd, and $\beta =p$. If $\alpha=p+2i$, $0<i\le \frac 12(p-1)$, then $(n,2;q)$ is desirable if and only if $4i\not\equiv 1\pmod p$; if $\alpha=p+2i-1$, $0<i\le \frac 12(p-1)$, then $(n,2;q)$ is desirable if and only if $4i\not\equiv 3\pmod p$. \cite[Theorems~5.6, 5.7]{FHL}
  \item\label{itm10} If $e=2$, $q$ is even, and $(\beta,\alpha)=(1,3)$, then $(n,2;q)$ is desirable if and only $q\equiv 1\pmod 3$. \cite[Theorem~5.9]{FHL} 
  \item  Assume that $e=2$, $q>2$, $(\beta,\alpha)=(1,2i)$, $i>0$. Then $(n,2;q)$ is desirable if and only if one of the following holds.
  \begin{itemize}  
    \item [(i)] $2i\equiv 1\pmod p$ and $q\equiv 1\pmod 4$;
    \item [(ii)] $2i\equiv -1\pmod p$ and $q\equiv\pm 1\pmod{12}$;
    \item [(iii)] $4i\equiv 1\pmod p$ and $q\equiv -1\pmod 6$.
  \end{itemize} 
\cite[Theorem~4.1]{Hou1}
  \item Assume that $e=2$, $q$ is odd, $(\beta,\alpha)=(1,2i+1)$, $i>0$. Then $(n,e;q)$ is desirable if and only $p\equiv 1$ or $3\pmod 8$, $q\equiv 1\pmod 8$, and $i^2=-\frac 12$. \cite[Corollary~6.1]{Hou2}
\end{enumerate}   

For $e\ge 3$, there was little activity as indicated by statement~\eqref{itm7}. For $e=2$, the situation appeared to be chaotic. In fact, computer search produced many desirable triples with $e=2$ (and $n=q^\alpha-q^\beta-1$) that are not covered by the above results; see \cite[Table 1]{FHL}. The case $(q^\alpha-q^\beta-1,2;q)$, which seemed hopeless till now, will be completely resolved in the next section.
When $n=q^\alpha-q^\beta-1$, $0\le \beta<\alpha$, the polynomial function $g_{n,q}(x)$ on $\Bbb F_{q^2}$ can be transformed into the form $Ax+Bx^q+Cx^{2q-1}$ through an invertible change of variable. Thus Theorems~A and B allow us to determine all desirable triples of the form $(q^\alpha-q^\beta-1,2;q)$, $0\le \beta<\alpha$. We note that for even $q$, all desirable triples $(q^\alpha-q^\beta-1,2;q)$ are already determined by a combination of some of the above statements, so Theorem~B is not necessary for this purpose.   

%%%%%%%%%%%%%%%%%%%%%%%%%%%%%%%%%%%%%%%%
%  section 6
%%%%%%%%%%%%%%%%%%%%%%%%%%%%%%%%%%%%%%%%

\section{Theorems~C and D}

\begin{lem}\label{L6.1} Assume $q>2$. Let $n=q^\alpha-q^\beta-1$, where $0<\beta<\alpha<2p$, $\beta$ is odd, and $\beta\ne p$. Write $\alpha-\beta=a_0+2a_1$, $0\le a_0\le 1$, and $\beta=1+2b_1$. Then
\begin{equation}\label{6.1} 
g_{n,q}(x)=A\phi(x)+B\phi(x)^q+C\phi(x)^{2q-1}\qquad\text{for all}\ x\in\Bbb F_{q^2},
\end{equation}
where $\phi$ is a permutation of $\Bbb F_{q^2}$ and
\begin{equation}\label{6.2}
\begin{cases}
\displaystyle A=\frac 1\beta(-a_0b_1+b_1+a_1),\cr
\displaystyle B=a_0-\frac{b_1+1}\beta,\cr
\displaystyle C=-\frac 1\beta(a_0b_1+a_0+a_1).
\end{cases}
\end{equation}
\end{lem}

\begin{proof}
For every integer $a\ge 0$, define $S_a={\tt x}+{\tt x}^q+\cdots+{\tt x}^{q^{a-1}}\in\Bbb F_p[{\tt x}]$. Let $x\in\Bbb F_{q^2}^*$. By \cite[(5.3)]{FHL},
\[
\begin{split}
g_{n,q}(x)\,&=-x^{q^2-2}-x^{q^2-q-2}\bigl(a_1S_2(x)+S_{a_0}(x)^q\bigr)\bigl((b_1S_2(x)+S_1(x))^{q-1}-1\bigr)\cr
&=-x^{-1}-x^{-q-1}\bigl(a_1x+(a_0+a_1)x^q\bigr)\bigl[\bigl((b_1+1)x+b_1x^q\bigr)^{q-1}-1\bigr]\cr
&=-y-\bigl(a_1y^q+(a_0+a_1)y\bigr)\bigl[\bigl((b_1+1)y^q+b_1y\bigr)^{q-1}-1\bigr],
\end{split}
\]
where $y=x^{-1}$. Note that $(b_1+1){\tt x}^q+b_1{\tt x}$ is a PP of $\Bbb F_{q^2}$ whose inverse on $\Bbb F_{q^2}$ is given by $\frac 1\beta\bigl((b_1+1){\tt x}^q-b_1{\tt x}\bigr)$. 

Let $z=(b_1+1)y^q+b_1y$. Then $y=\frac 1\beta\bigl((b_1+1)z^q-b_1z\bigr)$. We have 
\[
\begin{split}
g_{n,q}(x)=\,&-\frac 1\beta\bigl((b_1+1)z^q-b_1z\bigr)\cr
&=-\Bigl[a_1\frac 1\beta\bigl((b_1+1)z-b_1z^q\bigr)+(a_0+a_1)\frac 1\beta\bigl((b_1+1)z^q-b_1z\bigr)\Bigr](z^{q-1}-1)\cr
&=Az+Bz^q+Cz^{2q-1}.
\end{split}
\]
Extend the mapping $x\mapsto z$ to a bijection $\phi:\Bbb F_{q^2}\to\Bbb F_{q^2}$ by setting $\phi(0)=0$. Then \eqref{6.1} holds.
\end{proof}

\noindent{\bf Theorem C.} 
{\em Let $q$ be even and $n=q^\alpha-q^\beta-1$, where $0\le \beta<\alpha<2\cdot 2$. Then $(n,2;q)$ is desirable if and only if one of the following occurs.
\begin{itemize}
  \item [(i)] $q\equiv 1\pmod 3$, $(\beta,\alpha)=(0,2),\ (1,2),\ (1,3)$.
  \item [(ii)] $q=2$, $(\beta,\alpha)=(0,3)$.
\end{itemize}  
}

\begin{proof}
The conclusion follows from statements~\eqref{itm3}, \eqref{itm5}, \eqref{itm8}, \eqref{itm10} in Section~\ref{s5}.
\end{proof}

\noindent{\bf Theorem D.} 
{\em Let $q$ be odd and $n=q^\alpha-q^\beta-1$, where $0\le \beta<\alpha<2p$. Then $(n,2;q)$ is desirable if and only if one of the following occurs.
\begin{itemize}
  \item [(i)] $q\equiv 1\pmod 3$, $(\beta,\alpha)=(0,2)$.
  \item [(ii)] $\beta>0$, $\beta\equiv \alpha\equiv 0\pmod 2$.
  \item [(iii)] $(\beta,\alpha)=(p,p+i)$, $0<i<p$, $2i\not\equiv (-1)^i\pmod p$.
  \item [(iv)] $\beta\ne p$, $\beta=1+2b_1$, $\alpha-\beta=a_0+2a_1$, $a_0,a_1,b_1\in\Bbb N$, $0\le a_0\le 1$, and one of the following is satisfied.
  \begin{itemize}  
    \item [(iv.1)] $(a_1+b_1)(2a_1+b_1)+a_0(a_1-2a_1b_1-b_1^2)$ is a square in $\Bbb F_q^*$ and
\[
1+2b_1+2a_1^2+a_1b_1+a_0\bigl(-1-2b_1+b_1^2+a_1(3+2b_1)\bigr)\equiv 0\pmod p.
\]
    \item [(iv.2)] 
\[
\begin{cases}
a_0+2a_2+b_1\equiv 0\pmod p,\cr 
(1+b_1)^2-4a_1^2-a_0\bigl(5+10b_1+4b_1^2+8a_1(1+b_1)\bigr)\equiv 0\pmod p.
\end{cases}
\]
    \item [(iv.3)] 
\[
\begin{cases}
a_0=1,\ b_1=0,\cr 
4a_1+3\equiv 0\pmod p,\cr 
q\equiv -1\pmod 6.
\end{cases}
\]
    \item [(iv.4)] 
\[
\begin{cases}
a_0=1,\ a_1=0,\ b_1=0,\cr
q\equiv 1,3\pmod 6.
\end{cases}
\]
\end{itemize}   
\end{itemize}  
}

\begin{proof}
{\bf Case 1.} Assume $\beta=0$. We show that $(n,2;q)$ is desirable if and only if $\alpha=2$ and $q\equiv 1\pmod 3$. The ``if'' part follows from statement~\eqref{itm4} in Section~\ref{s5}. To prove the ``only if'' part, by statement~\eqref{itm4} in Section~\ref{s5}, it suffices to show that $(n,2;q)$ is not desirable for $\alpha>2$. 

Write $\alpha =a_0+2a_1$, $a_0,a_1\in\Bbb N$, $0\le a_0\le 1$. By \cite[(5.1)]{FHL}, for all $x\in\Bbb F_{q^2}$ we have 
\[
\begin{split}
g_{n,q}(x)\,&=x^{q-2}+x^{q^2-2}+\cdots+x^{q^{\alpha-1}-2}\cr
&= a_1(x^{q-2}+x^{q^2-2})+(a_0-1)x^{q^2-2}\cr
&=a_1x^{q-2}+(a_0+a_1-1)x^{q^2-2}\cr
&=(a_0+a_1-1)y^q+a_1y^{2q-1},
\end{split}
\]
where $y^q=x^{q^2-2}$. Note that $0<a_1<p$ and $0<a_0+a_1-1<p$, so $a_1\not\equiv 0\pmod p$ and $a_0+a_1-1\not\equiv 0\pmod p$. By Theorem~A, $(a_0+a_1-1){\tt x}^q+a_1{\tt x}^{2q-1}$ is a not PP of $\Bbb F_{q^2}$. So $g_{n,q}$ is not a PP of $\Bbb F_{q^2}$.

\medskip

{\bf Case 2.} Assume $\beta>0$ and $\beta\equiv 0\pmod 2$. By statement~\eqref{itm8} in Section~\ref{s5}, $(n,2;q)$ is desirable if and only if (ii) holds.

\medskip
{\bf Case 3.} Assume $\beta=p$. By statement~\eqref{itm9} in Section~\ref{s5}, $(n,2;q)$ is desirable if and only if (iii) holds.

\medskip
{\bf Case 4.} Assume $\beta\not\equiv 0\pmod 2$ and $\beta\ne p$. Write $\beta=1+2b_1$ and $\alpha-\beta=a_0+2a_1$, $a_0,a_1,b_1\in\Bbb N$, $0\le a_0\le 1$. By Lemma~\ref{L6.1}, $(n,2;q)$ is desirable if and only if $A{\tt x}+B{\tt x}^q+C{\tt x}^{2q-1}$ is a PP of $\Bbb F_{q^2}$, where $A,B,C$ are given by \eqref{6.2}. We claim that $C\ne 0$ in $\Bbb F_p$, i.e., $a_0b_1+a_0+a_1\not\equiv 0\pmod p$. In fact,
\[
0<a_0b_1+a_0+a_1\le a_1+b_1+1=\frac 12(1+2a_1+1+2b_1)\le\frac 12(\alpha-\beta+\beta)=\frac 12 \alpha<p.
\]
Thus $(n,2;q)$ is desirable if and only if $\frac AC{\tt x}+\frac BC{\tt x}^q+{\tt x}^{2q-1}$ is a PP of $\Bbb F_{q^2}$, which happens if and only if one of the conditions in Theorem~A holds with $a=\frac AC$ and $b=\frac BC$. Letting 
\[
a=\frac AC=\frac{a_0b_1-b_1-a_1}{a_0b_1+a_0+a_1},
\]
\[
b=\frac BC=\frac{-2a_0b_1-a_0+b_1+1}{a_0b_1+a_0+a_1}
\]
in Theorem~A, then conditions (i) -- (iv) in Theorem~A become (iv.1) -- (iv.4) in Theorem~D.
\end{proof}
 
%%%%%%%%%%%%%%%%%%%%%%%%%%%%%%%%%%%%%%%%
%   section 7
%%%%%%%%%%%%%%%%%%%%%%%%%%%%%%%%%%%%%%%%

\section{A Final Remark}

Let $f=a{\tt x}+b{\tt x}^q+{\tt x}^{2q-1}\in\Bbb F_q[{\tt x}]$, $ab\ne 0$. For $0\le s<q^2-1$, we saw in Section~\ref{s3.3} that $\sum_{x\in\Bbb F_{q^2}}f(x)^s=0$ unless $s=\alpha+\beta q$, $0\le\alpha,\beta\le q-1$, $\alpha+\beta=q-1$. 

Let $z=-\frac a{b^2}$ and assume that ${\tt x}^2+{\tt x}-z$ has two distinct roots in $\Bbb F_q$. By \eqref{3.16.0} and \eqref{3.17}, which hold for both odd and even $q$'s, we have
\begin{equation}\label{7.1}
\sum_{x\in\Bbb F_{q^2}}f(x)^{0+(q-1)q}=0,
\end{equation} 
\begin{equation}\label{7.2}
\sum_{x\in\Bbb F_{q^2}}f(x)^{1+(q-2)q}=\frac{2(1-a)(b^2-a^2-3a)}{a^2(b^2-4a)}.
\end{equation}
The sum $\sum_{x\in\Bbb F_{q^2}}f(x)^{2+(q-3)q}$ was computed in Subsection~\ref{s4.3} for even $q\ge 4$. That computation can be adapted for an arbitrary $q$ resulting in the following formula:
\begin{equation}\label{7.3}
\sum_{x\in\Bbb F_{q^2}}f(x)^{2+(q-3)q}=\frac{3b(1-a)(b^2-a^2-3a)(9a-6a^2+a^3-2b^2+ab^2)}{a^4(b^2-4a)^2},\quad q\ge 3.
\end{equation}
(The computation of \eqref{7.3}, which is quite lengthy and tedious, is given in the appendix.) Note that the sums \eqref{7.1} -- \eqref{7.3} are rational functions in $a$, $b$, independent of $q$, with coefficients in $\Bbb Z$. Moreover, the factor $b^2-a^2-3a$ appears in the numerator of each these three rational functions. In fact, this is true in general. Following the idea behind the computations in Subsections~\ref{s3.3} and \ref{s4.3}, it is not difficult to be convinced that for every $0\le \alpha\le q-1$, the sum $\sum_{x\in\Bbb F_{q^2}}f(x)^{\alpha+(q-1-\alpha)q}$ should be a rational function $R_\alpha(a,b)$ in $a,b$, independent of $q$, with coefficients in $\Bbb Z$, although we do not know the explicit expression of $R_\alpha(a,b)$ for a general $\alpha$. Since we already assumed that ${\tt x}^2+{\tt x}-z$ has two distinct roots in $\Bbb F_q$, by Theorems~A and B, the condition $b^2-a^2-3a=0$ implies that $f$ is a PP of $\Bbb F_{q^2}$, which further implies that $R_\alpha(a,b)=0$ for all $0\le \alpha\le q-1$. Hence $b^2-a^2-3a$ is a factor of the numerator of the reduced form of $R_\alpha(a,b)$. Two questions arise: Is it possible to compute $R_\alpha(a,b)$ explicitly for all $0\le\alpha\le q-1$? If $R_\alpha(a,b)$ is too complicated to compute, is there any more direct explanation why $b^2-a^2-3a$ always appears in the numerator of $R_\alpha(a,b)$?

\newpage
%%%%%%%%%%%%%%%%%%%%%%%%%%%%%%%%%%%%%%%%
%     appendix
%%%%%%%%%%%%%%%%%%%%%%%%%%%%%%%%%%%%%%%%

\appendix
\section{Computation of \eqref{7.3}}

Let $z=-\frac a{b^2}$. We have
\[
\begin{split}
&-\sum_{x\in\Bbb F_{q^2}}f(x)^{2+(q-3)q}\cr
=\,&\sum_{q-2+i+k-j-l=0,q+1}\binom 2i\binom ik\binom{q-3}j\binom jl a^{-i-j}b^{i+j-k-l}\cr
=\,&\sum_{q-2-j-l=0,q+1}\binom{q-3}j\binom jl a^{-j}b^{j-l}+\sum_{q-1-j-l=0,q+1}2\binom{q-3}j\binom jl a^{-1-j}b^{1+j-l}\cr
&+\sum_{q-j-l=0,q+1}2\binom{q-3}j\binom jl a^{-1-j}b^{j-l}+\sum_{q-j-l=0,q+1}\binom{q-3}j\binom jl a^{-2-j}b^{2+j-l}\cr 
&+\sum_{q+1-j-l=0,q+1}2\binom{q-3}j\binom jl a^{-2-j}b^{1+j-l}+\sum_{q+2-j-l=0,q+1}\binom{q-3}j\binom jl a^{-2-j}b^{j-l}\cr
=\,&\sum_{q-2-j-l=0}\binom{q-3}j\binom jl a^{-j}b^{j-l}+\sum_{q-1-j-l=0}2\binom{q-3}j\binom jl a^{-1-j}b^{1+j-l}\cr
&+\sum_{q-j-l=0}\binom{q-3}j\binom jl (2a^{-1-j}b^{j-l}+a^{-2-j}b^{2+j-l})\cr
&+\sum_{q+1-j-l=0,q+1}2\binom{q-3}j\binom jl a^{-2-j}b^{1+j-l}+2a^{-2}b\cr
&+\sum_{q+2-j-l=0,q+1}\binom{q-3}j\binom jl a^{-2-j}b^{j-l}+\sum_{j+l=1}\binom{q-3}j\binom jl a^{-2-j}b^{j-l}\cr
=\,&\sum_{1\le l\le\frac q2-1}\binom{q-3}{q-2-l}\binom{q-2-l}la^{-(q-2-l)}b^{q-2-2l}\cr
&+\sum_{2\le l\le\frac {q-1}2}2\binom{q-3}{q-1-l}\binom{q-1-l}la^{-1-(q-1-l)}b^{q-2l}\cr
&+\sum_{3\le l\le\frac q2}\binom{q-3}{q-l}\binom{q-l}l(2a^{-1-(q-l)}b^{q-2l}+a^{-2-(q-l)}b^{2+q-2l})\cr
&+\sum_{4\le l\le\frac {q+1}2}2\binom{q-3}{q+1-l}\binom{q+1-l}la^{-2-(q+1-l)}b^{2+q-2l}+2a^{-2}b\cr
&+\sum_{5\le l\le\frac q2+1}\binom{q-3}{q+2-l}\binom{q+2-l}la^{-2-(q+2-l)}b^{q+2-2l}-3a^{-3}b\cr
=\,&\sum_{1\le l\le\frac q2-1}(-1)^{l+1}\binom{q-l}2\binom{-2-l}la^{1+l}b^{-1-2l}\cr
&+\sum_{2\le l\le\frac {q-1}2}2(-1)^l\binom{q+1-l}2\binom{-1-l}la^{-1+l}b^{1-2l}\cr
\end{split}
\]
\[
\begin{split}
&+\sum_{3\le l\le\frac q2}(-1)^{l+1}\binom{q+2-l}2\binom{-l}l(2a^{-2+l}b^{1-2l}+a^{-3+l}b^{3-2l})\cr
&+\sum_{4\le l\le\frac {q+1}2}2(-1)^l\binom{q+3-l}2\binom{1-l}la^{-4+l}b^{3-2l}\cr
&+\sum_{5\le l\le\frac q2+1}(-1)^{l+1}\binom{q+4-l}2\binom{2-l}l a^{-5+l}b^{3-2l}+2a^{-2}b-3a^{-3}b\cr
=\,&-\frac ab\sum_{1\le l\le\frac q2-1}\binom{-l}2\binom{-l-2}l z^l
+2\frac ba \sum_{2\le l\le\frac {q-1}2}\binom{-l+1}2\binom{-l-1}l z^l\cr
&-\Bigl(2\frac b{a^2}+\frac{b^3}{a^3}\Bigr)\sum_{3\le l\le\frac q2}\binom{-l+2}2\binom{-l}l z^l
+2\frac{b^3}{a^4}\sum_{4\le l\le\frac {q+1}2}\binom{-l+3}2\binom{-l+1}l z^l\cr
&-\frac{b^3}{a^5}\sum_{5\le l\le\frac q2+1}\binom{-l+4}2\binom{-l+2}l z^l+2\frac b{a^2}-3\frac b{a^3}\cr
=\,&\frac ab\sum_{1\le l\le\frac q2-1}\binom{-l}2\binom{-l-1}{l+1} z^l+4\frac ba \sum_{2\le l\le\frac {q-1}2}\binom{-l+1}2\binom{-l}l z^l\cr
&-\Bigl(2\frac b{a^2}+\frac{b^3}{a^3}\Bigr)\biggl[\,\sum_{0\le l\le\frac q2}\binom{-l+2}2\binom{-l}l z^l-1\,\biggr]-2\frac b{a^3}\sum_{3\le l\le\frac {q-1}2}\binom{-l+2}2\binom{-l}{l+1} z^l\cr
&-\frac{b^3}{a^5}\sum_{5\le l\le\frac q2+1}\Bigl[\frac 12 l(l+5)-6(l-1)\Bigr]\binom{-l+2}l z^l+2\frac b{a^2}-3\frac b{a^3}\cr
=\,&-b\sum_{2\le l\le\frac q2}\binom{-l+1}2\binom{-l}l z^l+4\frac ba \sum_{2\le l\le\frac q2}\binom{-l+1}2\binom{-l}l z^l\cr
&-\Bigl(2\frac b{a^2}+\frac{b^3}{a^3}\Bigr)\sum_{0\le l\le\frac q2}\binom{-l+2}2\binom{-l}l z^l+2\frac b{a^2}+\frac{b^3}{a^3}\cr
&-2\frac b{a^3}\sum_{0\le l\le\frac {q-1}2}\binom{-l+2}2\binom{-l}{l+1} z^l\cr
&-\frac{b^3}{a^5}\sum_{5\le l\le\frac q2+1}\Bigl[\frac{(l+5)(-l+2)}2\binom{-l+1}{l-1}+3(-l+2)\binom{-l+1}l\Bigr]z^l+2\frac b{a^2}-3\frac b{a^3}\cr
=\,&\Bigl(-b+4\frac ba\Bigr)\sum_{0\le l\le\frac q2}\binom{-l+1}2\binom{-l}l z^l-\Bigl(2\frac b{a^2}+\frac{b^3}{a^3}\Bigr)\sum_{0\le l\le\frac q2}\binom{-l+2}2\binom{-l}l z^l\cr
&-2\frac b{a^3}\sum_{0\le l\le\frac {q-1}2}\binom{-l+2}2\binom{-l}{l+1} z^l\cr
&+\frac b{a^4}\sum_{4\le l\le\frac q2}\Bigl[\frac{(l+6)(-l+1)}2\binom{-l}l+3(-l+1)\binom{-l}{l+1}\Bigr]z^l+4\frac b{a^2}+\frac{b^3}{a^3}-3\frac b{a^3}\cr
\end{split}
\]
\[
\begin{split}
=\,&\Bigl(-b+4\frac ba\Bigr)\sum_{0\le l\le\frac q2}\binom{-l+1}2\binom{-l}l z^l-\Bigl(2\frac b{a^2}+\frac{b^3}{a^3}\Bigr)\sum_{0\le l\le\frac q2}\binom{-l+2}2\binom{-l}l z^l\cr
&-2\frac b{a^3}\sum_{0\le l\le\frac {q-1}2}\binom{-l+2}2\binom{-l}{l+1} z^l\cr
&+\frac b{a^4}\biggl[\sum_{0\le l\le\frac q2}\frac{(l+6)(-l+1)}2\binom{-l}l z^l-3+12z^2-90z^3\biggr]\cr
&+3\frac b{a^4}\biggl[\sum_{0\le l\le\frac {q-1}2}(-l+1)\binom{-l}{l+1}z^l-4z^2+30z^3\biggr]+4\frac b{a^2}+\frac{b^3}{a^3}-3\frac b{a^3}\cr
=\,&\Bigl(-b+4\frac ba\Bigr)\sum_{0\le l\le\frac q2}\Bigl[\binom{l+2}2-2(l+1)+1\Bigr]\binom{-l}lz^l\cr
&-\Bigl(2\frac b{a^2}+\frac{b^3}{a^3}\Bigr)\sum_{0\le l\le\frac q2}\Bigl[\binom{l+2}2-3(l+1)+3\Bigr]\binom{-l}l z^l\cr
&-2\frac b{a^3}\sum_{0\le l\le\frac {q-1}2}\Bigl[\binom{l+2}2-3(l+1)+3\Bigr]\binom{-l}{l+1} z^l\cr
&-\frac b{a^4}\sum_{0\le l\le\frac q2}\Bigl[\binom{l+2}2+(l+1)-5\Bigr]\binom{-l}l z^l-3\frac b{a^4}\cr
&-3\frac b{a^4}\sum_{0\le l\le\frac {q-1}2}\bigl[(l+1)-2\bigr]\binom{-l}{l+1}z^l+4\frac b{a^2}+\frac{b^3}{a^3}-3\frac b{a^3}\cr
=\,&\Bigl(-b+4\frac ba-2\frac b{a^2}-\frac{b^3}{a^3}-\frac b{a^4}\Bigr)\sum_{0\le l\le\frac q2}\binom{l+2}2\binom{-l}lz^l\cr
&+\Bigl(2b-8\frac ba+6\frac b{a^2}+3\frac{b^3}{a^3}-\frac b{a^4}\Bigr)\sum_{0\le l\le\frac q2}(l+1)\binom{-l}lz^l\cr
&+\Bigl(-b+4\frac ba-6\frac b{a^2}-3\frac{b^3}{a^3}+5\frac b{a^4}\Bigr)\sum_{0\le l\le\frac q2}\binom{-l}lz^l-2\frac b{a^3}\sum_{0\le l\le\frac {q-1}2}\binom{l+2}2\binom{-l}{l+1} z^l\cr
&+\Bigl(6\frac b{a^3}-3\frac b{a^4}\Bigr)\sum_{0\le l\le\frac {q-1}2}(l+1)\binom{-l}{l+1} z^l+\Bigl(-6\frac b{a^3}+6\frac b{a^4}\Bigr)\sum_{0\le l\le\frac {q-1}2}\binom{-l}{l+1} z^l\cr
&+4\frac b{a^2}+\frac{b^3}{a^3}-3\frac b{a^3}-3\frac b{a^4}.
\end{split}
\]
Eq.~\eqref{7.3} is obtained by making the following substitutions.
\[
\sum_{0\le l\le\frac q2}\binom{-l}lz^l=1, \kern 3.8cm \text{(Lemma~\ref{L3.1})}
\]
\[
\sum_{0\le l\le\frac q2}(l+1)\binom{-l}lz^l=\frac{1+3z}{1+4z},\kern 2cm \text{(Lemma~\ref{L3.3})}
\]
\[
\sum_{0\le l\le\frac q2}\binom{l+2}2\binom{-l}lz^l=\frac{1+6z+11z^2}{(1+4z)^2}\qquad q>2, \kern 1cm\text{(Lemma~\ref{L4.2})}
\]
\[
\sum_{0\le l\le\frac {q-1}2}\binom{-l}{l+1}z^l=1, \kern 5.2cm\text{(Lemma~\ref{L3.2})}
\]
\[
\sum_{0\le l\le\frac {q-1}2}(l+1)\binom{-l}{l+1}z^l=\frac{2z}{1+4z},\kern 3.5cm\text{(Lemma~\ref{L3.2})}
\]
\[
\sum_{0\le l\le\frac {q-1}2}\binom{l+2}2\binom{-l}{l+1}z^l=\frac{3z(1+2z)}{(1+4z)^2}, \qquad q>2. \kern 1cm\text{(Lemma~\ref{LA1})}
\]

\begin{lem}\label{LA1}
Let $z\in\Bbb F_q^*$ and assume that ${\tt x}^2+{\tt x}-z$ has two distinct roots in $\Bbb F_q$. Then 
\[
\sum_{0\le l\le \frac{q-1}2}\binom{l+2}2\binom{-l}{l+1}z^l=
\begin{cases}
0&\text{if}\ q=2,\vspace{2mm}\cr
\displaystyle\frac{3z(1+2z)}{(1+4z)^2}&\text{if}\ q>2.
\end{cases}
\]
\end{lem}

\begin{proof}
The case $q=2$ is trivial, and we assume $q>2$. We have 
\[
\begin{split}
S:\,&=\sum_{0\le l\le \frac{q-1}2}\binom{l+2}2\binom{-l}{l+1}z^l\cr
&=\sum_{0\le l\le q-3}\binom{l+2}2\binom{-l}{l+1}z^l\cr
&=\sum_{0\le l\le q-3}\binom{l+2}2\cdot\text{ct}\Bigl(\frac 1{{\tt x}^{l+1}(1+{\tt x})^l}\Bigr)\cdot z^l\cr
&=\text{ct}\biggl[\frac 1{\tt x}\sum_{0\le l\le q-3}\binom{l+2}2\Bigl(\frac z{{\tt x}(1+{\tt x})}\Bigr)^l\biggr]\cr
&=-z\cdot \text{ct}\Bigl[{\tt x}^{-q-1}\Bigl(1+\frac z{{\tt x}^2+{\tt x}-z}\Bigr)^3\Bigr]\kern 2cm \text{(see \eqref{4.6})}.
\end{split}
\]
Write ${\tt x}^2+{\tt x}-z=({\tt x}-r_1)({\tt x}-r_2)$ and let $c=\frac 1{r_1-r_2}$. By \eqref{4.7} and \eqref{4.8}, we have 
\[
\begin{split}
S=\,&-z\biggl[(3zc-6z^2c^3+6z^3c^5)\bigl((-r_1)^{-q-2}-(-r_2)^{-q-2}\bigr)\binom{-1}{q+1}\cr
&+(3z^2c^2-3z^3c^4)\bigl((-r_1)^{-q-3}+(-r_2)^{-q-3}\bigr)\binom{-2}{q+1}\cr
&+z^3c^3\bigl((-r_1)^{-q-4}-(-r_2)^{-q-4}\bigr)\binom{-3}{q+1}\biggr]\cr
=\,&-z\bigl[(3z-6z^2c^2+6z^3c^4)\cdot c\cdot(-r_1^{-3}+r_2^{-3})+(3z^2c^2-3z^3c^4)(r_1^{-4}+r_2^{-4})\cdot 2\cr
&+z^3c^2\cdot c\cdot(-r_1^{-5}+r_2^{-5})\cdot 3 \bigr] \kern 2cm \text{($\textstyle\binom{-3}{q+1}\equiv 3\pmod p$ for $q>2$)},
\end{split}
\]
where 
\[
c^2=\frac 1{1+4z},
\]
\[
c(r_1^{-3}-r_2^{-3})=\frac 1{r_1-r_2}\cdot\frac{r_2^3-r_1^3}{(r_1r_2)^3}=\frac{-(r_1^2+r_1r_2+r_2^2)}{(-z)^3}
=\frac{(r_1+r_2)^2-r_1r_2}{z^3}=\frac{1+z}{z^3},
\]
\[
r_1^{-4}+r_2^{-4}=\frac{r_1^4+r_2^4}{(r_1r_2)^4}=\frac{(r_1^2+r_2^2)^2-2(r_1r_2)^2}{z^4}
=\frac{(1+2z)^2-2z^2}{z^4}=\frac{1+4z+2z^2}{z^4},
\]
\[
\begin{split}
c(r_1^{-5}-r_2^{-5})\,&=\frac 1{r_1-r_2}\cdot\frac{r_2^5-r_1^5}{(r_1r_2)^5}\cr
&=\frac{-(r_1^4+r_1^3r_2+r_1^2r_2^2+r_1r_3^3+r_2^4)}{(-z)^5}\cr
&=\frac{(r_1^2+r_2^2)^2-r_1^2r_2^2+r_1r_2(r_1^2+r_2^2)}{z^5}\cr
&=\frac{(1+2z)^2-z^2-z(1+2z)}{z^5}\cr
&=\frac{1+3z+z^2}{z^5}.
\end{split}
\]
The final expression for $S$ follows from the above substitutions.
\end{proof} 

%%%%%%%%%%%%%%%%%%%%%%%%%%%%%%%%%%%%%%%%
%          references 
%%%%%%%%%%%%%%%%%%%%%%%%%%%%%%%%%%%%%%%% 

\end{document}